\documentclass[12pt,a4paper]{amsart}

\usepackage{amsfonts}
\usepackage{amsmath}
\usepackage{amssymb}
\usepackage{amsthm}
\usepackage{mathrsfs}
\usepackage{color}

\usepackage[normalem]{ulem}

\usepackage{graphicx}
\usepackage[all]{xy}

\usepackage{enumerate}

\usepackage[colorlinks]{hyperref}

\newcommand{\f}{\varphi}

\newcommand{\aA}{{\mathcal{A}}}
\newcommand{\bB}{{\mathcal{B}}}
\newcommand{\cC}{{\mathcal{C}}}
\newcommand{\dD}{{\mathcal{D}}}

\newcommand{\eE}{{\mathcal{E}}}
\newcommand{\fF}{{\mathcal{F}}}
\newcommand{\gG}{{\mathcal{G}}}

\newcommand{\lL}{{\mathcal{L}}}

\newcommand{\oO}{{\mathcal{O}}}

\newcommand{\sS}{{\mathcal{S}}}
\newcommand{\tT}{{\mathcal{T}}}

\newcommand{\tr}{{\mbox{$t$-struc}\-tu\-r}}

\newcommand{\aff}{\mathop{\textrm{aff}}\nolimits}
\newcommand{\Hom}{\mathop{\textrm{Hom}}\nolimits}
\newcommand{\Ext}{\mathop{\operatorname{Ext}}\nolimits}
\newcommand{\End}{\mathop{\textrm{End}}\nolimits}

\newcommand{\Coh}{\mathop{\operatorname{Coh}}\nolimits}
\newcommand{\QCoh}{\mathop{\textrm{QCoh}}\nolimits}
\newcommand{\Lex}{\mathop{\textrm{Lex}}\nolimits}
\newcommand{\Rex}{\mathop{\textrm{Rex}}\nolimits}
\newcommand{\Id}{\mathop{\textrm{Id}}\nolimits}

\newcommand{\opp}{\mathop{\textrm{op}}\nolimits}
\newcommand{\Ab}{\mathop{\mathcal{A}\textrm{b}}\nolimits}

\newcommand{\Spec}{\mathop{\textrm{Spec}}\nolimits}
\newcommand{\Rref}{\mathop{\operatorname{Ref}}\nolimits}
\newcommand{\Pic}{\mathop{\operatorname{Pic}}\nolimits}
\newcommand{\Aut}{\mathop{\operatorname{Aut}}\nolimits}
\newcommand{\Auteq}{\mathop{\operatorname{Auteq}}\nolimits}

\newcommand{\wt}[1]{\widetilde{#1}}
\newcommand{\wh}[1]{\widehat{#1}}
\newcommand{\ol}[1]{\overline{#1}}

\newtheorem{LEM}{Lemma}[section]

\newtheorem*{THM*}{Theorem}
\newtheorem{THM}[LEM]{Theorem}
\newtheorem{PROP}[LEM]{Proposition}
\newtheorem{COR}[LEM]{Corollary}

\theoremstyle{definition}
\newtheorem{EXM}[LEM]{Example}
\newtheorem{REM}[LEM]{Remark}

\begin{document}
\title{Reconstruction of a surface from the category of reflexive sheaves}

\today
	\author{Agnieszka Bodzenta}
\author{Alexey Bondal}
\date{\today}

\address{Agnieszka Bodzenta\\
	Institute of Mathematics, 
	University of Warsaw \\ Banacha 2 \\ Warsaw 02-097,
	Poland} \email{A.Bodzenta@mimuw.edu.pl}

\address{Alexey Bondal\\
	Steklov Mathematical Institute of Russian Academy of Sciences, Moscow, Russia, and \\
	Centre for Pure Mathematics, Moscow Institute of Physics and Technology, Russia, and\\
	Kavli Institute for the Physics and Mathematics of the Universe (WPI), The University of Tokyo, Kashiwa, Chiba 277-8583, Japan} 
\email{bondal@mi-ras.ru}

	\begin{abstract}
		We define a normal surface $X$ to be codim-2-saturated  if any open embedding of $X$  into a normal surface with the complement of codimension 2 is an isomorphism.
   We show that  any normal surface $X$ allows a codim-2-saturated model $\widehat{X}$ together with the canonical open embedding $X\to \widehat{X}$. 

    Any normal surface which is proper over its affinisation is codim-2-saturated, but the converse does not hold.
    We give a criterion for a surface to be codim-2-saturated in terms of its Nagata compactification and the boundary divisor.  
    
    We reconstruct the codim-2-saturated model of a normal surface $X$ from the additive category of reflexive sheaves on $X$. We show that the category of reflexive sheaves on $X$ is quasi-abelian  and we use its canonical exact structure for the reconstruction.  
    
	In order to deal with categorical issues, we introduce a class of weakly localising Serre subcategories in abelian categories. These are Serre subcategories whose categories of closed objects are quasi-abelian. This general technique might be of independent interest.	
\end{abstract}
	\maketitle
	\tableofcontents
	
	\section{Introduction}
In this paper we address the problem of reconstructing 
a normal surface $X$ (a scheme, not an algebraic space) from the category $\Rref(X)$ of reflexive sheaves on it. The important issue here is that the category is considered as an additive category, not endowed with the (geometric) exact structure induced by the embedding into the category of coherent sheaves on $X$. In the case of smooth surfaces, reflexive sheaves are just vector bundles. The typical instance of the problem is how to recover a projective plane ${\mathbb P}^2$ from the category of vector bundles without using any additional data.

It is clear that the surface $X$ cannot be uniquely reconstructed if no condition is imposed on it. For instance, one can remove a closed point from $X$ and the category  $\Rref(X)$ would not change.

For this reason, we reconstruct the maximally saturated model $\wh{X}$ of $X$. 
The surface $X$ is an open subset of $\wh{X}$ with complement of dimension zero  and $\wh{X}$ is maximal with this property, i.e. it does not allow `adding' one more closed point while keeping the surface normal and separated. We say that a normal surface is \emph{codim-2-saturated} if it has this property.   Any proper surface is such, as well as any surface which is proper over its affinisation, though the inverse implication does not hold.
 
The key point in our reconstruction is that $\Rref(X)$ is {\it quasi-abelian}, hence it is endowed with the canonical exact structure. Remarkably, this exact structure is not geometric, but rather corresponds to a `surface' with all closed points removed. We use the right abelian envelope of the canonical exact structure to categorically identify the Weil prime divisors of the surface. 

To recover $\wh{X}$ from $\Rref(X)$ we find those divisors whose complements are quasi-affine surfaces, using the fact that categories equivalent to $\Rref(X)$ for $X$ quasi-affine can be characterised intrinsically (we call them CM affine categories). Then $\wh{X}$ is basically glued from affinisations of the quasi-affine open pieces that cover $X$.

Actually,  the construction of $\wh{X}$ out of $X$ can be performed in purely geometric terms. 
We prove that the process of adding closed points to a separated normal surface while preserving the property of it being normal and separated must terminate after a finite number of steps.
All these points can be added simultaneously giving the codim-2-saturated model $\wh{X}$ of $X$. 
The explicit construction of $\wh{X}$ is via the affinisations: 
\begin{equation}\label{eqtn_intro}
\wh{X} =\varinjlim_{U\in\mathcal{QA}(X)} U^{\aff},
\end{equation}	
where the colimit is taken over the category $\mathcal{QA}(X)$
of quasi-affine open subsets $U$ of $X$.

The canonical open embedding $X\to \wh{X}$ turns out to be the unit for a pair of adjoint functors.
We consider the category $\mathscr{D}$ of normal surfaces and open embeddings with complement of codimension 2. Morphisms in $\mathscr{D}$ form a multiplicative system, so we consider the groupoid 
$\wh{\mathscr{D}} := \mathscr{D}[\mathscr{D}^{-1}]$ by inverting all morphisms in $\mathscr{D}$. 
We prove
\begin{THM*}(\emph{cf.} Theorem \ref{thm_adj_unit})
	Functor $\mathfrak{s}$ admits the right adjoint $\mathfrak{s}_*$. For any $X\in \mathscr{D}$, we have $\wh{X} =\mathfrak{s}_*\mathfrak{s}(X)$ and the open embedding $j_X \colon X\to \wh{X} $ is the adjunction unit.
\end{THM*}

A normal surface $X$ is codim-2-saturated if $X\to \wh{X}$ is an isomorphism, or, equivalently, if any morphism in $\mathscr{D}$ with domain $X$ is an isomorphism.
As it was already mentioned above, 
any normal surface which is proper over its affinisation is codim-2-saturated, but the converse is not true. 

Actually, the converse holds in the category of algebraic spaces, but recovering a normal algebraic space of dimension two from its category of reflexive sheaves is beyond the scope of this paper. Instead, we characterize codim-2-saturated surfaces as open subschemes in proper surfaces via a geometric criterion on contractibility of connected components of the complement divisor.

The codim-2-saturated model $\wh{X}$ of $X$ clearly determines $\Rref(X)$, because the pullback along  $X\to \wh{X}$ yields an equivalence $\Rref(\wh{X}) \xrightarrow{\simeq} \Rref(X)$. Conversely, we show that  $\wh{X}$ can be recovered from the (additive) category $\Rref(X)$.  We 
recover $\Coh(\wh{X})$ from $\Rref(X)$ and then use the standard argument to reconstruct $\wh{X}$ from $\Coh(\wh{X})$.

The canonical exact structure on the quasi-abelian category $\Rref(X)$ can be described as follows: conflations are complexes $F\to G\to H$, such that $0 \to F\to G \to H \to T \to 0$ is an exact sequence in $\Coh(X)$ with $T$ in  the Serre subcategory $\Coh_{\leq 0}(X)$ of Artinian sheaves. 
We prove
\begin{THM*}(\emph{cf.} Corollary \ref{cor_env_of_E_n-2})
	The category $\Coh^{\geq 1}(X) :=\Coh(X)/\Coh_{\leq 0}(X)$ is the right abelian envelope $\aA_r(\Rref(X))$ of $\Rref(X)$ endowed with the canonical exact structure.
\end{THM*}

The category $\Coh^{\geq 1}(X)$ has a (cotiliting) torsion pair $\tau_X=(\bigoplus_{D \in X^1} O_{X,D}\textrm{-mod}, \Rref(X))$, where $X^1$ is the set of 1-dimensional points of $X$, $O_{X,D}$ the local ring of a prime divisor $D \in X^1$ and $O_{X,D}\textrm{-mod}$ is the category of finite length $\oO_{X,D}$-modules. The category $\bigoplus_{D \in X^1} O_{X,D}\textrm{-mod}$ is the subcategory of finite length objects in $\aA_r(\Rref(X))$ with simple objects $\oO_D$,  
where $D$ runs over $X^1$. 
Hence, the torsion pair $\tau_X$ is canonical with the torsion part generated by all simple objects in $\Coh^{\geq 1}(X)$.

This allows us to interpret prime divisors in $X$ as simple objects in the right abelian envelope of $\Rref(X)$. By taking a finite set $I$ of simple objects and moding out of $\Coh^{\geq 1}(X)$ the Serre subcategory $\sS_I$ generated by them, we get the quotient category $\Coh^{\geq 1}(X\setminus D_I)$ of coherent sheaves on the open subset $X\setminus D_I$, where $D_I=\bigcup_{i \in I} D_i$ is the divisor corresponding to $I$. 
We define  the category $\fF_I$ as the right orthogonal to all simple objects in $\Coh^{\geq 1}(X)/\sS_I$.  We show that $\fF_I$ is equivalent to $\Rref(X\setminus D_I)$.

The affinisation of a normal surface $X$ can be easily read off $\Rref(X)$. Indeed, $H^0(X, \oO_X)$ is the center of $\Rref(X)$. Recall that the center of an additive category $\eE$ is 
the endomorphism algebra of the identity endofunctor, {\it i.e.} $Z(\eE) = \End(\Id_{\eE})$.

In order to mimic formula (\ref{eqtn_intro}) in terms of the categories $\Rref(X\setminus D_I)$, we need to characterise quasi-affine normal surfaces in terms of the categories of reflexive sheaves on them.
To this end, we define  \emph{CM affine} categories. These are the categories $\eE$ equivalent to $\Rref(\Spec Z(\eE))$ via the functor $\Hom(L, -)$, for an object $L$ in $\eE$ with $\End(L) =Z(\eE)$.
\begin{THM*}(\emph{cf.} Theorem \ref{thm_CM_aff=q-aff})
	A normal surface $X$ is quasi-affine if and only if $\Rref(X)$ is CM affine.
\end{THM*}

Now we are ready to recover $\Coh(\wh{X})$:

	\begin{THM*}(\emph{cf.} Theorem \ref{thm_final_model_from_Ref})
		Let $X$ be a normal surface. Then 
	\begin{equation}\label{eqtn_intro_2}
\Coh(\wh{X})\simeq \underleftarrow{\textrm{2-lim} }\,
\textrm{mod--}Z(\fF_I), 
\end{equation}
where the 2-limit is taken over the category of finite sets $I$ of simple objects in $\aA_r(\Rref(X))$ for which the category $\fF_I$ defined above is CM affine. 
\end{THM*} 
Above,  $\textrm{mod-}Z(\fF_I)$ denotes the category of finitely generated modules over the center $Z(\fF_I)$ of $\fF_I$. 

In Section \ref{ssec_cont_of_X}, we introduce a geometric model $\ol{\eE}$ of what we call a height-one quasi-abelian category $\eE$. $\Rref(X)$ on a normal surface $X$ is an example of such a category.
\begin{THM*}(\emph{cf.} Corollary \ref{cor_reconst_of_X})
	A codim-2-saturated surface $X$ is isomorphic to $\ol{\Rref(X)}$.
\end{THM*}
	
We develop a general categorical set-up which $\Rref(X)$ fits in as a subcategory of $\Coh(X)$. We believe this is of independent interest and can be applied to other relevant (noncommutative) situations. 

Given a Serre subcategory $\bB \subset \aA$, the subcategory $\eE$ of $\bB$-closed objects in $\aA$ is classically known. This category is fully faithfully embedded by the quotient functor into $\aA / \bB$. We say that $\bB$ is a {\it weakly localising} Serre subcategory in $\aA$ if $\aA / \bB$ possesses a torsion pair with the torsion free part being $\eE$. We introduce the embedding condition (\ref{eqtn_ast}) for the torsion free part $\fF$ of a torsion pair in $\aA$ which allows us to induce a required torsion pair in $\aA / \bB$ from the torsion pair in $\aA$. The model geometric example here is the case when $\aA =\Coh (X)$, $\bB =\Coh_{\leq 0}(X)$, $\eE =\Rref(X)$, and $\fF$ the category of torsion free sheaves on $X$.

This paper can be regarded as the contemplation on exact structures on the category $\Rref(X)$. The embedding $\Rref(X) \subset \Coh(X)$ as an extension closed subcategory induces the \emph{geometric exact structure} on $\Rref(X)$ in which conflations are complexes $F\to G \to H$ such that $0 \to F\to G \to H \to 0$ is an exact sequence in $\Coh(X)$. For any non-isomorphism $i\colon X \to Y$ in $\mathscr{D}$ the geometric exact structures on $\Rref(X)$ induced by $\Coh(X)$ and $\Coh(Y)$ are different, the one induced by $Y$ has less conflations. Hence, the geometric exact structure corresponding to a codim-2-saturated surface is minimal among possible geometric exact structures, while the canonical exact structure can be thought of as the maximal one among those, when all the closed points from the surface are removed. Note that there are other interesting exact structures which are related to various non-commutative resolutions of $\Coh(X)$. We postpone the study of them to another publication.

We expect that the extension of ideas of this paper to higher dimensions will involve the machinery of the Minimal Model Program from Birational Geometry. 

\noindent 
\textbf{Structure of the paper.}
In Section \ref{sec_codim_2_sat_model} we introduce the categories $\mathscr{D}$ and its groupoid $\wh{\mathscr{D}} := \mathscr{D}[\mathscr{D}^{-1}]$. We define the codim-2-saturated model $\wh{X}$ of a normal surface $X$, see (\ref{eqtn_def_of_X_hat}), and prove that the canonical embedding $j_X\colon X\to \wh{X}$ is a morphism in $\mathscr{D}$, see Theorem \ref{thm_prop_of_X_hat}. We show that $j_X$ is the adjunction unit $X\to \mathfrak{s}_* \mathfrak{s}(X)$, for the canonical functor $\mathfrak{s}\colon \mathscr{D}\to \wh{\mathscr{D}}$ and its right adjoint, see Theorem \ref{thm_adj_unit}. We also characterise a  codim-2-saturated surfaces as appropriate open subsschemes in proper surfaces, see Theorem \ref{thm_open_in_prop_final}.

The category $\eE$ of closed objects for a Serre subcategory  $\bB\subset \aA$ is studied in Section \ref{sec_closed_obj}. We show that $\eE$ is a full subcategory of the quotient $\aA/\bB$, see Corollary \ref{cor_ff_on_closed}. After discussing when a torsion pair in $\aA$ induces a torsion pair in $\aA/\bB$, see Section \ref{ssec_tor_pair_on_quotient}, we give a condition for $\eE$ to be the torsion-free part of a torsion pair in $\aA/\bB$, i.e. for the category $\bB$ to be \emph{weakly localising}, see Theorem \ref{thm_emb_cond}. Given a pair $E_1$, $E_2$ of objects in $\eE$, we describe the abelian group $\Ext^1_{\aA/\bB}(E_1, E_2)$ as the colimit of  Ext-groups in $\aA$, see Proposition \ref{prop_Ext_as_colim}.

In Section \ref{sec_ref_sheaves} we study the category $\Rref(X)$ of reflexive sheaves on a normal surface $X$. We show that the subcategory $\Coh_{\leq 0}(X)\subset \Coh(X)$ of Artinian sheaves is weakly localising and the quotient $\Coh^{\geq 1}(X) = \Coh(X)/\Coh_{\leq 0}(X)$ is the right abelian envelope of the quasi-abelian category $\Rref(X)$ of $\Coh_{\leq 0}(X)$ closed objects, see Theorem \ref{thm_tor_pair_for_eqidim_scheme} and Corollary \ref{cor_env_of_E_n-2}. We prove that $\Coh^{\geq 1}(X)$ admits a torsion pair with torsion-free part $\Rref(X)$ and the Serre subcategory generated by all simple objects as the torsion part, see Proposition \ref{prop_canon_torsion_pair_Coh_1}.

We further discuss that $\Coh(X)$ is the right abelian envelope of $\Rref(X)$ endowed with the geometric exact structure. This allows us to conclude that $H^0(X, \oO_X)$ is the center of $\Rref(X)$, see Proposition \ref{prop_cent_of_E_0}. We prove that $\Rref(X) \simeq \Rref(Y)$ if and only if $X$ and $Y$ are isomorphic outside of dimension zero, see Theorem \ref{thm_iso_iff_E_n-2}, and, given an open subset $U\subset X$, we describe the category $\Rref(U)$ in terms of $\Rref(X)$ and its right abelian envelope, see Proposition \ref{prop_Ref_U_from_Ref_X}.

In Section \ref{sec_reconstr} we introduce CM affine categories discussed above and prove that $\Rref(X)$ is CM affine if and only if the normal surface $X$ is quasi-affine, see Theorem \ref{thm_CM_aff=q-aff}. Finally, we reconstruct the category $\Coh(\wh{X})$ of coherent sheaves on a codim-2-saturated model of a normal surface $X$ from $\Rref(X)$, i.e. we prove equivalence (\ref{eqtn_intro_2}), see Theorem \ref{thm_final_model_from_Ref}.

We recall the notion of an exact category and its right abelian envelope in Appendix \ref{sec_ab_env_of_q-a_cat}. We discuss torsion pairs in abelian categories and argue that the passage from a cotilting torsion pair to its torsion-free part gives a correspondence between torsion pairs and quasi-abelian categories. 

In the Appendix \ref{sec_X_from_coh_X} we recall, following \cite{Gabriel}, the reconstruction of a Noetherian scheme $Z$ from the abelian category $\Coh(Z)$.

\noindent
\textbf{Acknowledgements.}
We are indebted to Joachim Jelisiejew, Karol Szumi{\l}o, Jaros{\l}aw Wi{\'s}niewski and Yulya Zaitseva for useful discussions. The first named author was partially supported by Polish National Science Centre grants No. 2018/31/D/ST1/03375, 2018/29/B/ST1/01232 and 2021/41/B/ST1/03741. 
The reported study was funded by RFBR and CNRS, project number 21-51-15000. This work was supported by JSPS KAKENHI grant number JP20H01794.
This research was partially supported by World Premier International Research Center Initiative (WPI Initiative), MEXT, Japan.

\noindent
\textbf{Notation.}
We work over an arbitrary ground field $k$.

\section{The codim-2-saturated model of a normal surface}\label{sec_codim_2_sat_model}

Throughout this section $X$  is a normal surface, i.e. an irreducible two dimensional separated normal $k$-scheme of finite type.

Define the category $\mathscr{D}$ whose objects are normal surfaces and morphisms are open inclusions with complement of codimension two.
We say that a normal surface $X$ is \emph{codim-2-saturated} if one cannot add any closed points to it, i.e.  if any morphism in $\mathscr{D}$ with domain $X$ is an isomorphism.

Any normal surface $X$ admits a canonical morphism to a codim-2-saturated surface $\wh{X}$ which is the adjunction unit for the functor $\mathscr{D}\to \wh{\mathscr{D}}$, where $\wh{\mathscr{D}}$ is the groupoid obtained from $\mathscr{D}$ by inverting all morphisms.

\vspace{0.3cm}
\subsection{The affinisation of a normal surface}~\\

We denote by $X^{\aff} = \Spec H^0(X, \oO_X)$ the \emph{affinisation} of $X$ and by $\rho_X \colon X\to X^{\aff}$ the canonical morphism.
By \cite[Corollary 6.3]{Schroer}, $H^0(X, \oO_X)$ is an integrally closed $k$-algebra of finite type, hence $X^{\aff}$ is a normal variety of dimension 0,1 or 2. 

Note that the affinisation $X^{\aff}$ has the universal property for morphisms to affine schemes. In particular, a morphism $X\to Y$ induces a unique morphism $X^{\aff} \to Y^{\aff}$.

\begin{PROP}\label{prop_aff_of_quasi-affine}
	Let $X$ be a quasi-affine normal surface. Then 
	\begin{enumerate}
		\item the affinisation $\rho_X \colon X\to X^{\aff}$ is an open embedding into a normal scheme with complement of codimension two.
		\item Given an open embedding $j\colon U \to X$ of a quasi-affine $U$, 	
		$j^{\aff} \colon U^{\aff} \to X^{\aff}$ is an open embedding. 
	\end{enumerate}
\end{PROP}
\begin{proof}
	Let $U \subset X\subset Y$ be open embeddings with $Y$ affine.
	By replacing $Y$ with its normalisation, if necessary, we can assume that $Y$ is normal. Let $D_1\subset Y$, $D_2\subset Y$ be the divisorial parts of $Y \setminus X$, respectively of $Y\setminus U$. By  \cite[Theorem 5]{Nagata1} \emph{cf.} \cite[Theorem 3.4]{Brenner}, the open subsets $Y_1 = Y \setminus D_1$ and $Y_2 = Y \setminus D_2$ are affine normal surfaces. Since $D_1 \subset D_2$, $Y_2$ is an open subset of $Y_1$. As $X\subset Y_1$ and $U \subset Y_2$ are open subsets with complement of codimension two, a version of Hartogs' Lemma for normal surfaces
	implies that $X^{\aff} \simeq Y_1$ and $U^{\aff} \simeq Y_2$. The statement follows.
\end{proof}

\begin{PROP}\label{prop_aff_is_sat}
	An affine normal surface is codim-2-saturated.
\end{PROP}
\begin{proof}
Let $X$ be an affine surface and $j\colon X\to Z$ a morphism in $\mathscr{D}$. By Hartogs' Lemma, the affinisations of $X$ and $Z$ coincide. Since $X$ is affine, it
 is the affinisation of $Z$, and the composition of $j$ with $\rho_Z \colon Z\to Z^{\aff} \simeq X$ is $\Id_X$. Hence, $j$ is an open embedding with a left inverse, i.e. an isomorphism.
\end{proof}

\vspace{0.3cm}
\subsection{Inverting open embeddings}~\\

We will consider the category $\wh{\mathscr{D}}$ obtained from $\mathscr{D}$ by inverting all morphisms. 

Recall that a class $\mathscr{S}$ of morphisms in a category $\cC$ is a \emph{left multiplicative system} if 
\begin{itemize}
	\item[(LMS 1)] The identity on every object in $\cC$ is in $\mathscr{S}$ and the composition of two composable morphisms in $\mathscr{S}$ is in $\mathscr{S}$.
	\item[(LMS 2)] Every  diagram of solid arrows
	\[
	\xymatrix{A \ar[r]^f \ar[d]_s & B\ar@{-->}[d]^t\\ C \ar@{-->}[r]_g & D }
	\]
	with $s \in \mathscr{S}$ can be completed by dashed arrows to a commutative square with $t\in \mathscr{S}$.
	\item[(LMS 3)] Given a pair of morphisms $f,g\colon C_1 \to C_2$ such that $f\circ s = g\circ s$, for some $s\in \mathscr{S}$ with codomain $C_1$, there exists $t\in \mathscr{S}$ with domain $C_2$ such that $t\circ f= t \circ g$.
\end{itemize}
A left multiplicative system $\mathscr{S}$ allows one to consider the category $\cC[\mathscr{S}^{-1}]$ whose objects are objects of $\cC$ and morphisms are equivalence classes of `roofs': 
\[
\xymatrix{&D' & \\
C\ar[ur]^f & & D \ar[ul]_s}
\]
with $s\in \mathscr{S}$. We denote such a roof by $s^{-1}f$.

Two roofs $C\rightarrow D' \leftarrow  D$, $C \rightarrow D'' \leftarrow D$ are equivalent if there exists a roof $C\rightarrow\wt{D}  \leftarrow D$ which  fits into a commutative diagram
\begin{equation}\label{eqtn_equiv_of_roofs}
\xymatrix{& D' \ar[d] & \\
	C\ar[r] \ar[ur] \ar[dr] & \wt{D}  & \ar[l] \ar[ul] \ar[dl] D\\
	& D''\ar[u]  & }
\end{equation}

Given a right multiplicative system $\mathscr{S}_r$, the category $\cC[\mathscr{S}_r^{-1}]$ is defined analogously.

\begin{PROP}\label{prop_left_ms}
	All morphisms in $\mathscr{D}$ form a left multiplicative system.
\end{PROP}
\begin{proof}
	Condition (LMS 1) is clear. Since the image of a morphism in $\mathscr{D}$ is a dense open subset, (LMS 3) holds. Indeed, let $f,g\colon X_1 \to X_2$ be morphisms such that $f\circ s = g\circ s$, for some open embedding $s\colon X\to X_1$. Then $f =g$, in particular $\Id_{X_2} \circ f = \Id_{X_2} \circ g$.
	
	It remains to consider a diagram of solid arrows 
		\[
		\xymatrix{X \ar[r]^{i_1} \ar[d]_{i_2} & X_1 \ar@{-->}[d]^{j_1} \\ X_2 \ar@{-->}[r]_{j_2} & \ol{X}}
	\]
	and complete it  by dashed arrows to a commutative square.
	
We define $\ol{X}$ by gluing affine open sets. Then we check that it is a normal surface.

For a (closed) point $x\in X_1 \setminus X$, let $U_x \subset X_1$ be an affine open neighbourhood such that $U_x^o:=U_x\setminus \{x\} \subset X$. 
Consider the domain $\wt{X}$ of the birational map $i_1\circ i_2^{-1} \colon X_2 \dashrightarrow X_1$. We define $V_x\subset \wt{X} \subset X_2$ as the preimage of $U_x\subset X_1$. Then $V_x$ is open in $X_2$ and 
admits a morphism $\kappa \colon V_x \to U_x$ extending the inclusion $U_x^o \to U_x$.

Since $V_x\subset X_2$, the domain of $\kappa$ is separated, hence so is $\kappa$,
see \cite[Proposition 9.13]{GorWed}. Further, as $V_x$ is of finite type and $U_x$ is separated, $\kappa$ is of finite type, see \cite[Proposition 10.7]{GorWed}. Since the complement of $X$ in $\wt{X}$ consists of finitely many points, 
so does the complement of $U_x^o$ in $V_x$. It follows that $\kappa^{-1}(x)$ is a finite set. Finally, Hartogs' Lemma 
implies that $\kappa_*\oO_{V_x} \simeq \oO_{U_x}$, hence, by \cite[Theorem 12.83]{GorWed}, $\kappa\colon V_x \to U_x$ is an open immersion, i.e. we have open subsets $U_x^o \subset V_x \subset U_x$ such that $U_x \setminus U_x^o = \{x\}$. It follows that either $V_x \simeq U_x^o$ or $V_x\simeq U_x$.

The set $\mathcal{V}:=X\cup \bigcup_{x\in X_1\setminus X} V_x$ is open in $X_2$. For any $y \in X_2\setminus \mathcal{V}$, let $V_y \subset X_2$ be an affine open neighbourhood such that $V_y^o:=V_y \setminus \{y\} \subset X$.

The set $\mathcal{U}:=\bigcup_{x\in X_1\setminus X} U_x^o \cup \bigcup_{y\in X_2\setminus \mathcal{V}} V_y^o$ is open in $X$. Fix affine open $W_1,\ldots, W_r\subset X$ such that $X = \mathcal{U} \cup W_1 \cup \ldots \cup W_r$. We define $\ol{X}$ as the gluing of $\bigsqcup_{x\in X_1\setminus X} U_x \sqcup \bigsqcup_{y\in X_2 \setminus \mathcal{V}} V_y \sqcup \bigsqcup W_i$ along the open subsets defined as the intersections of $U_x^o$, $V_y^o$ and $W_i$ in $X$.

By construction, $\ol{X}$ is a Noetherian normal surface of finite type over $\Spec k$. It remains to check that $\ol{X}$ is separated. We use the valuative criterion for separatedness \cite[Proposition 7.2.3]{EGAII}. Since $\ol{X}$ is a Noetherian scheme, $\ol{X} \to \Spec (k)$ is quasi-separated \cite[Corollary 10.24]{GorWed}. Let now $R$ be a discrete valuation ring with the field of fractions $K=\textrm{Frac} R$.  Assume that $\xi_1,\xi_2 \colon \Spec R \to \ol{X}$ restrict to $\eta \colon \Spec K \to \ol{X}$.

Denote by $x_1$ and $x_2$ the images of the closed point of $\Spec R$ under $\xi_1$ and $\xi_2$. Since $X_1$ and $X_2$ are separated, we can assume that $x_1 \in \ol{X} \setminus X_2$ and $x_2\in \ol{X} \setminus X_1$. 	Let $U = \Spec A \subset X_1$ be the affine open neighbourhood of $x_1$ that we used in the definition of $\ol{X}$. As before, let $U^o:=U\setminus \{x\}$ be an open subset in $X$.

We claim that $V:=U^o \cup \{x_2\}$ is an open subset in $X_2$. Assume this is not the case. Then the closed set $X_2\setminus U^o$ contains a 1-dimensional component passing through $x_2$. Indeed, if $x_2$ is an isolated point of $X_2\setminus U^o$, there exists an open set $W\subset X_2$ such that $W\cap (X_2\setminus U^o) = \{x_2\}$. It follows that $W\cup U^o = U^o\cup \{x_2\}$ is open in $X_2$, which contradicts our assumption.

Let $C_2\subset X_2$ be the curve in $X_2\setminus U^o$ such that $\{x_2\}\in C_2$. Then $C = C_2\cap S$ is a closed curve in $X$, hence $X\setminus C$ is an open subset in $X_1$. It follows that $U\cap (X\setminus C) \neq \emptyset$, hence also $U^o \cap (X\setminus C) \neq \emptyset$ and $U^o \cap (X_2\setminus C_2) \neq \emptyset$. The contradiction with the definition of $C_2$ implies that $V\subset X_2$ is indeed an open subset.

Hartogs' Lemma  
implies that $\Gamma(V, \oO_V) =A$, hence the affinisation gives a map $V\to U$, which commutes with the inclusion of $U^o$. It follows that $V\simeq U$ is an open subset of $X_2$. Therefore, $\xi_1$ and $\xi_2$ factor via the inclusion of $U$ into $\ol{X}$. As $U$ is separated, $\xi_1=\xi_2$, i.e. $\ol{X}$ is separated.	
\end{proof}

Proposition \ref{prop_left_ms} allows us to consider the category $\wh{\mathscr{D}}:=\mathscr{D}[\mathscr{D}^{-1}]$. We denote by $\mathfrak{s}\colon \mathscr{D} \to \wh{\mathscr{D}}$ the canonical functor.

It is easy to check that  all morphisms in $\mathscr{D}$ form a right multiplicative system too. We choose to prove that they form a left multiplicative system as property (LMS 2) will be useful in the construction of the right adjoint $\mathfrak{s}_*\colon \wh{\mathscr{D}} \to \mathscr{D}$.

\vspace{0.3cm}
\subsection{The codim-2-saturated model of $X$}~\\

We will check that $\mathfrak{s} \colon \mathscr{D}\to \wh{\mathscr{D}}$ admits the right adjoint. To define it we introduce the codim-2-saturated model of $X$.

A normal surface $X$ admits a  quasi-affine open cover, 
$$
X = \varinjlim U,
$$
where the colimit is taken over the category $\mathcal{QA}(X)$ of quasi-affine open subsets of $X$ and open inclusions.

Proposition \ref{prop_aff_of_quasi-affine} allows us to consider the scheme
\begin{equation}\label{eqtn_def_of_X_hat}
\wh{X} =\varinjlim_{\mathcal{QA}(X)} U^{\aff}.
\end{equation}	

Indeed, if $U$ and $V$ are quasi-affine and $U\subset V$ is open, then so is $U^{\aff} \subset V^{\aff}$, hence (\ref{eqtn_def_of_X_hat}) is a gluing of schemes along open subschemes, i.e. a scheme itself, \emph{cf}. \cite[01JC]{stacks-project}.

The open embeddings $U \to U^{\aff}$, see Proposition \ref{prop_aff_of_quasi-affine}, yield a natural transformation of functors on $\mathcal{QA}(X)$, hence a morphism 
\begin{equation} \label{eqtn_incl_to_sat_model}
j_X \colon X \to \wh{X}.
\end{equation}

We check that $\wh{X}$ is a normal codim-2-saturated surface and $j_X$ is a morphism in $\mathscr{D}$.

Note that for a quasi-affine $U$,  $\wh{U} = U^{\aff}$ and $j_U = \rho_U$ is the affinisation morphism.

Surface $X$ admits two morphisms to affine schemes, the affinisation morphism $\rho_X \colon X\to X^{\aff}$ and the structure morphism $\pi \colon X\to \Spec k$.
Since $X$ is separated and of finite type over $k$, so are $\rho_X$ and $\pi$ \cite[Propositions 9.13, 10.7]{GorWed}. Hence, by Nagata's compactification theorem, $\rho_X$ and $\pi$ admit  decompositions 
\begin{equation}\label{eqtn_Nagata}
\xymatrix{X \ar[r]^i \ar[dr]_{\rho_X} & Y \ar[d]^p && X \ar[r]^\alpha \ar[dr]_{\pi} & \ol{Y} \ar[d]^{\beta} \\ & X^{\aff} & &&	 	 \Spec k}
\end{equation}
into open immersions $i$, $\alpha$ and proper morphisms $p$, $\beta$ \cite[Theorem 12.70]{GorWed}. By considering the normalisation, if necessary, one can assume that both $Y$ and $\ol{Y}$ are normal surfaces.

For an open quasi-affine $U\subset X$ and the corresponding $\wh{U} \subset \wh{X}$ denote by $Y_U$, resp. $\ol{Y}_U$, the domain of the birational map $f_U\colon Y \dashrightarrow \wh{U}$, resp. $\ol{f}_U \colon \ol{Y} \dashrightarrow \wh{U}$, which restricts to the identity on $U$ considered as an open subset of both $Y$, resp. $\ol{Y}$, and $\wh{X}$.

\begin{PROP}\label{prop_f_U_proper}
	The morphisms $f_U \colon Y_U \to \wh{U}$, $\ol{f}_U \colon \ol{Y}_U \to \wh{U}$ are proper.
\end{PROP}
\begin{proof}
	Let $p_1$ and $p_2$ denote the canonical projections from $Y\times_{X^{\aff}} \wh{U}$ to  $Y$ and $\wh{U}$. The graph of $f_U$ is the closure $Z$ of the diagonal embedding of $U$ into $Y\times_{X^{\aff}} \wh{U}$. 
	
	For any $y\in Y\setminus U $ the projection $p_2(p_1^{-1}(y) \cap Z)$ is contained in $\wh{U} \setminus U$. Indeed, if it is not the case then $f_U^{-1} \colon \wh{U} \dashrightarrow Y$ would not be the identity on $U$. It follows that $p_1|_Z \colon Z\to Y$ is a quasi-finite birational morphism. Since $Y$ is normal, the Zariski's main theorem \cite[Corollary 12.88]{GorWed} implies that $Z\to Y$ is an open inclusion. Hence, $Z = Y_U\subset Y$ is the domain of $f_U$. It follows that $f_U \colon Y_U \to Y\times_{X^{\aff}} \wh{U} \xrightarrow{p_2} \wh{U}$ is the composition of the closed embedding of $Z$ and a proper morphism $p_2$. Hence, $f_U$ is proper.
	
	An analogous proof of the properness of $\ol{f}_U$ is omitted.	
\end{proof}

Let $V\subset Y$, respectively $\ol{V}\subset \ol{Y}$ be the domain of the birational map $Y\dashrightarrow \wh{X}$, resp. $\ol{Y} \dashrightarrow \wh{X}$, which restricts to the identity on $X$. Since $\wh{X}$ is covered by affinisations of quasi-affine open subsets of $X$,  Proposition \ref{prop_f_U_proper} implies that 
$$
f\colon V\to \wh{X}, \quad \ol{f}\colon \ol{V} \to \wh{X}
$$ 
are proper.

\begin{PROP}\label{prop_finite_cov_of_X_hat}
	The scheme $\wh{X}$ admits a finite open covering $\wh{X} = \bigcup_{i=1}^n \wh{U_i}$, for open quasi-affine $U_i\subset X$.
\end{PROP}
\begin{proof}
	By construction, $\wh{X}$ is a union of $\wh{U}$, where $U$ runs over open quasi-affine subsets in $X$.
	Since $V$ is quasi-compact, there exists a finite open subcovering $\{f^{-1}(\wh{U_1}), \ldots, f^{-1}(\wh{U_n})\}$ of $V$. As the birational morphism $f$ is proper, hence surjective, $\{\wh{U_1}, \ldots, \wh{U_n}\}$ is an affine open covering of $\wh{X}$.
\end{proof}

\begin{THM}\label{thm_prop_of_X_hat}
	Consider a normal surface $X$. Then $\wh{X}$ is a separated normal surface which admits $X$ as an open subset with complement of codimension two, i.e. $j_X\colon X\to \wh{X}$ is a morphism in $\mathscr{D}$.
\end{THM}
\begin{proof}
	By \cite[Corollary 6.3]{Schroer}, the affinisation of any normal surface is again normal. Hence, $\wh{X}$ is normal. By Proposition \ref{prop_aff_of_quasi-affine}, for any $U\subset X$ open quasi-affine, $\wh{U}\setminus U$ is a union of finitely many points. Since $\wh{X}$ is covered by finitely many open subset of the form $\wh{U}$ (see Proposition \ref{prop_finite_cov_of_X_hat}), $\wh{X}$ is of finite type over $k$ and $\wh{X}\setminus X$ is also a finite union of closed points, i.e. a closed subset of $\wh{X}$.
	Since $X$ is irreducible, so is $\wh{X}$. It remains to check that $\wh{X}$ is separated. To this end we shall show that $\wh{\rho_X} \colon \wh{X} \to X^{\aff}$ is separated. 
	
	We shall use the valuative criterion \cite[Theorem 15.9]{GorWed}. Fix a discrete valuation ring $R$ with the field of fractions $K$ and consider 
	$v\colon \Spec R \to X^{\aff}$ such that $v|_{\Spec K}$ decomposes as $\wh{\rho_X} \circ u$:
	\[
	\xymatrix{& V \ar[r]^i \ar[d]^f & Y \ar[ddl]^p \\
		\Spec K \ar[r]^u \ar[d] \ar[ur]^{\wt{u}}& \wh{X} \ar[d]^{\wh{\rho_X}} & \\
		\Spec R \ar[r]^v \ar@{-->}[ur]^{\wt{v}_i}& X^{\aff}.&}
	\]
	Assume that $\wt{v}_1, \wt{v}_2 \colon \Spec R\to \wh{X}$ are morphisms such that $\wh{\rho_X} \circ \wt{v}_1 = v = \wh{\rho_X} \circ \wt{v}_2$. Since $j_X\colon X\to \wh{X}$ is an open embedding with complement of codimension two, morphism $u$ admits a decomposition $u = j_X\circ u'$, for some $u' \colon \Spec K \to X$. Consider the composite $\wt{u} \colon \Spec K \xrightarrow{u'}X \to V$. Morphism $f$ is proper, hence there exist $w_1, w_2 \colon \Spec R \to V$ such that $f\circ w_i = \wt{v}_i$.  Then $p \circ i \circ w_1 = v = p \circ i \circ w_2$. The properness of $p$ implies that $i\circ w_1 = i\circ w_2$. As $i$ is an open embedding, $w_1 = w_2$, hence $\wt{v}_1 = f\circ w_1 = f\circ w_2 = \wt{v}_2$.
\end{proof}

\begin{PROP}\label{prop_equiv_cond_for_X=X_hat}
	For a normal surface $X$, $j_X\colon X\to \wh{X}$ is an isomorphism if and only if for any quasi-affine open subset $i \colon U\to X$  its affinisation $\wh{U}$ admits an open embedding $i'\colon \wh{U} \to X$ such that $i'|_U = i$. 
\end{PROP}
\begin{proof}
	If $\wh{U}$ is an open subset of $X$, for any quasi-affine
	open $U\subset X$, then the definition (\ref{eqtn_def_of_X_hat}) of $\wh{X}$ implies that $\wh{X} \subset X$.
	
	In the opposite direction, let $j_X \colon X\to \wh{X}$ be an isomorphism. Consider $U\subset X$ quasi-affine. Then $\wh{U}$ is an open subset of $\wh{X} \simeq X$.
\end{proof}

\begin{PROP}\label{prop_X_hat_sat}
	Given a normal surface $X$, $\wh{X}$ is codim-2-saturated.
\end{PROP}
\begin{proof}
	Let $j \colon \wh{X} \to Y$ be a morphism in $\mathscr{D}$. An affine open $U\subset Y$ is the affinisation of the quasi-affine $U\cap \wh{X}$. Hence, by Proposition \ref{prop_equiv_cond_for_X=X_hat}, $U$ admits an open embedding into $ \wh{X}$. Since $U$ was arbitrary, $Y$ is an open subset of $\wh{X}$. Hence, $j$ is an isomorphism.
\end{proof}
\begin{PROP}\label{prop_map_to_X_hat}
	Given a morphism $j\colon X\to Y$ in $\mathscr{D}$, there exists a unique $i\colon Y\to\wh{X}$ such that $i \circ j = j_X$.
\end{PROP}
\begin{proof}
	By Proposition \ref{prop_left_ms}, the diagram of solid arrows 
	\[
	\xymatrix{X\ar[r]^j \ar[d]_{j_X} & Y \ar@{-->}[d]^f \\ \wh{X} \ar@{-->}[r]_g & Z}
	\]
	can be completed to a commutative square with dashed arrows added. By Proposition \ref{prop_X_hat_sat}, $g$ is an isomorphism. Hence, $g^{-1}\circ f\colon Y \to \wh{X}$ is the sought morphism. 
	
	Morphism $i$ is unique because it is fixed on a dense open subset $X\subset Y$ with complement of codimension two. 
\end{proof}

\vspace{0.3cm}
\subsection{The right adjoint $\mathfrak{s}_* \colon \wh{\mathscr{D}} \to \mathscr{D}$}~\\

A roof $s^{-1} f$ in $\Hom_{\wh{\mathscr{D}}}(X, Y)$ gives a commutative diagram
\begin{equation}\label{eqtn_morphism_from_roof}
\xymatrix{\wh{X} & \wh{Y}\ar[l]_g & \\
X \ar[u]^{j_X} \ar[r]_f & Z \ar[u]_j & Y \ar[l]_s \ar[ul]_{j_Y}}
\end{equation}
where $j$, respectively $g$, are the morphisms given by Proposition \ref{prop_map_to_X_hat} for $s\colon Y \to Z$, respectively $j\circ f\colon X\to \wh{Y}$. Note that, by Proposition \ref{prop_X_hat_sat}, $g$ is an isomorphism. 

Consider functor $\mathfrak{s}_*\colon \wh{\mathscr{D}} \to \mathscr{D}$ which maps an object $X$ to $\wh{X}$ and a roof $s^{-1}f\colon X\to Y$ to $g^{-1}$, where $g\colon \wh{Y} \to \wh{X} $ is as in (\ref{eqtn_morphism_from_roof}).

It is immediate to check that $\mathfrak{s}_*$ is well-defined on equivalence classes of roofs and $\mathfrak{s}_*$ is indeed a functor.

\begin{THM}\label{thm_adj_unit}
	Functor $\mathfrak{s}_*$ is right adjoint to $\mathfrak{s}\colon \mathscr{D} \to \wh{\mathscr{D}}$. For any $X\in \mathscr{D}$, $j_X \colon X\to \wh{X} = \mathfrak{s}_*\mathfrak{s}(X)$ is the adjunction unit.
\end{THM}
\begin{proof}
	We check that for any $X, Y \in \mathscr{D}$,
	$$
	i\colon \Hom_{\wh{\mathscr{D}}}(X,Y) \to \Hom_{\mathscr{D}}(X, \wh{Y}),\quad s^{-1}f \mapsto g^{-1}j_X = jf
	$$
	is a bijection, where $j$ and $g$ are as in (\ref{eqtn_morphism_from_roof}).
	
	The inverse of $i$ maps $h\in \Hom_{\mathscr{D}}(X, \wh{Y})$ to $j_{Y}^{-1}h$. Indeed, $i^{-1}i(s^{-1}f) = i^{-1}(jf) = j_Y^{-1}jf$ is a roof equivalent to $s^{-1}f$ while $ii^{-1}(h) = i(j_Y^{-1}h) = \wt{g}^{-1}j_X = h$, where $\wt{g} \colon \wh{Y} \to \wh{X}$ is the unique map such that $\wt{g}h = j_X$.
\end{proof}

\begin{PROP}
	Functor $\mathfrak{s}_* \colon \wh{\mathscr{D}}\to \mathscr{D}$ is fully faithful.
\end{PROP}
\begin{proof}
	By \cite[Corollary 1.23]{Huy} it suffices to check that the adjunction counit $\mathfrak{s} \mathfrak{s}_* \to \Id$ is an isomorphism when applied to any object in $\wh{\mathscr{D}}$. This is clear, as any morphism in $\wh{\mathscr{D}}$ is an isomorphism.
\end{proof}

	\vspace{0.3cm}
\subsection{Criteria for codim-2-saturatedness}~\\

We give several characterisations of codim-2-saturated surfaces. Since $\wh{X}$ is codim-2-saturated and $j_X\colon X\to \wh{X}$ is an open embedding, surface $X$ is saturated if and only if $j_X$ is an isomorphism.

\begin{THM}\label{thm_proper_is_final}
	A normal surface $X$ proper over its affinisation $X^{\aff}$ is codim-2-saturated. 
\end{THM}
\begin{proof}
	It suffices to check that any closed point $x$ of $\wh{X}$ is a point of $X$. Let $v\colon \Spec R \to \wh{X}$ be the germ of a curve via $x$. Then the restriction of $v$ to the filed of fractions $K$ of $R$ decomposes as $j_X \circ u$, for some $u \colon \Spec K \to X$. The properness of $\rho_X \colon X\to X^{\aff}$ implies that there exists a lift $\wt{v}\colon \Spec R \to X$ of $\wh{\rho_X} \circ v$ along $\rho_X$, i.e. $\rho_X\circ \wt{v} = \wh{\rho_X} \circ v$. 
	\[
	\xymatrix{\Spec K \ar[r]^u \ar[d] & X \ar[r]^{j_X} \ar[d]|(0.3){\rho_X} & \wh{X} \ar[dl]^{\wh{\rho_X}} \\ \Spec R \ar[r]_{\wh{\rho_X} \circ v} \ar[ur]^{\wt{v}} \ar[urr]|(0.4){v} 
		& X^{\aff}.&}
	\]
	In particular, $\wh{\rho_X} \circ v = \wh{\rho_X} \circ j_X \circ \wt{v}$ . As $\wh{\rho_X}$ is separated (see Theorem \ref{thm_prop_of_X_hat}), $v = j_X \circ \wt{v}$, in particular, the point $x$ lies in the image of $j_X$.
\end{proof}

\begin{EXM}\label{exm_exm_Hironaka}
	A codim-2-saturated normal surface $X$ is not necessarily proper over its affinisation.
	The counterexample is given over $\mathbb{C}$ by the Nagata-Mumford example of the complement $X$ to the strict transform of an elliptic curve $C\subset \mathbb{P}^2$ under the blow up of $\mathbb{P}^2$ in 10 points on $C$ in general position. The strict transform of $C$ has self-intersection $-1$, hence it can be contracted in the category of normal algebraic spaces \cite[Corollary 6.12(b)]{Art5}, but not schemes \cite[Section 4]{Art_alg_st}. By Theorem \ref{thm_open_in_prop_final} below, $X$ is codim-2-saturated, but it is not proper over its affinisation $\Spec \mathbb{C}$.	
\end{EXM}

Consider a proper normal surface $Y$ and a connected 
 curve $D\subset Y$. We say that $D$ is \emph{contractible} if there exists a normal surface $Y_D$ (a scheme) and a contraction $Y\to Y_D$, i.e. a proper birational morphism which contracts $D$ to a point and is an isomorphism on $Y\setminus D$.

We call an open subset $X\subset Y$ a \emph{non-contractible divisorial complement}  if $Y\setminus X$ is a divisor in $Y$ and any connected component of $Y \setminus X$ is not contractible. 

\begin{THM}\label{thm_open_in_prop_final}
	A normal surface $X$ is codim-2-saturated if and only if $X$ is isomorphic to a non-contractible divisorial complement in a normal surface $Y$ proper over $\Spec k$.
\end{THM}
\begin{proof}
	Let $\ol{Y}$ be a compactification of $X$ as in (\ref{eqtn_Nagata}). Let $\ol{Y} \setminus X = D\sqcup Z$ be the decomposition of the complement of $X$ into the divisorial part $D$ and the Artinian part $Z$. If $Z\neq \emptyset$, let $U_z$ be an affine open neighbourhood of $z\in Z$ such that $V_z:=U_z\setminus \{z\} \subset X$. Then $V_z$ is a quasi-affine open subset of $X$ with affinisation $U_z$. By Proposition \ref{prop_equiv_cond_for_X=X_hat}, $U_z$ is an open subset of $X$ which contradicts the choice of $z\in \ol{Y}\setminus X$. Hence, $Z= \emptyset$, i.e. $X$ is a divisorial complement in $\ol{Y}$. 
	
	Let $D\subset \ol{Y}\setminus X$ be a negative definite connected component. If $g_D\colon \ol{Y} \to Y_D$ is the contraction of $D$ then the complement to $X$ in $Y_D$ has non-trivial Artinian part. Indeed, as $D$ is a connected component of $\ol{Y}\setminus X$, the point $g_D(D)\notin X$ does not lie on any of the divisorial components of $Y_D\setminus X$. Since $Y_D$ is a compactification of $X$, an analogous argument as above yields a contradiction with the assumption that $X$ is codim-2-saturated. Hence, $D$ is not contractible.
	
	In the opposite direction, assume that $X\subset \ol{Y}$ is a non-contractible divisorial complement. In particular, $\ol{Y}$ is a compactification of $X$ as in (\ref{eqtn_Nagata}). Let $\wh{U}$ be the affinisation of a quasi-affine open subset $U\subset X$ and let $\ol{Y}_U$ be the domain of the rational map $\ol{Y} \dashrightarrow \wh{U}$ which restricts to the identity on $U$. By Proposition \ref{prop_f_U_proper}, $\ol{f}_U \colon \ol{Y}_U \to \wh{U}$ is proper. 
	
	Assume that, for a point $p\in \wh{U}\setminus U$, the fiber of $\ol{f}_U$ over $p$ is of dimension one, i.e. $\ol{f}_U^{-1}(p) =D$ is a proper connected curve contained in $\ol{Y}\setminus X$. Proposition \ref{prop_aff_of_quasi-affine} implies that $U\cup \{p\}$ is an open subset of $\wh{U}$. Hence, $X\cup D = \ol{f}_U^{-1}(U \cup \{p\})$ is an open subset of $Y$. It follows that $D$ is a connected component of $\ol{Y}\setminus X$. Indeed, if this is not the case, there exists a point  on $D$ whose any open neighbourhood is not contained in $X\cup D$. Hence, $\ol{f}_U$ contracts a connected component of $\ol{Y}\setminus X$ which contradicts the assumption on $X$.
	
	It follows that, for any $p\in \wh{U}\setminus U$ the fiber of $\ol{f}_U$ over $p$ is a point. As $X\subset \ol{Y}$ is a divisorial complement, $\ol{f}_U^{-1}(p)$ is a point of $X$. Hence, $\wh{U}\simeq \ol{f}_U^{-1}(\wh{U})$ is an open subset of $X$. By Proposition \ref{prop_equiv_cond_for_X=X_hat}, $X$ is codim-2-saturated.
\end{proof}

Recall that a curve $D = \bigcup D_i$ in a proper normal surface $Y$ is said to be \emph{negative definite} if the intersection matrix $|D_i \cdot D_j|$ is negative definite.

In view of 
\begin{THM}\cite[Theorem 3.3]{Schroer}
	A connected negative definite curve $D\subset Y$ is contractible if and only if there is a curve $A\subset Y$ disjoint from $D$ such that $A\cdot C>0$ for every curve $C\subset Y$ not supported by $D$.
\end{THM}
Theorem \ref{thm_open_in_prop_final} can be reformulated as the following geometric statement
\begin{COR}\label{cor_final_num_cond}
	A normal surface $X$ is codim-2-saturated if and only if $X$ is isomorphic to an open subset in a normal surface $Y$ proper over $\Spec k$ such that $Y\setminus X$ is a divisor in $Y$ and, for any negative definite connected component $D$ of $Y\setminus X$ and for any curve $A\subset Y$ disjoint from $D$, there exists a curve $C\subset Y$ not supported by $D$ such that $A\cdot C \leq 0$.
\end{COR}

\section{Closed objects for a Serre subcategory}\label{sec_closed_obj}

Throughout this section we consider a Serre subcategory $\bB$ in an abelian category $\aA$ and the full subcategory 
\begin{equation}\label{eqtn_def_of_closed_obj}
\eE = \{A\in \aA\,|\, \Hom(\bB, A) = 0 = \Ext^1(\bB, A)\}
\end{equation}
of $\bB$-{\it closed objects} \cite[Lemme III.2.1.b]{Gabriel}.

We denote by
\begin{equation}\label{eqtn_def_quot_funct}
q\colon \aA\to \aA/\bB =: \cC
\end{equation} 
the quotient functor.

We show that $\eE$ is a full subcategory in $\cC$ and discuss when it is the torsion-free part of a torsion pair, i.e. when $\bB$ is weakly localising. To this end, we introduce the embedding condition (\ref{eqtn_ast}) which allows us to induce a torsion pair in $\cC$ from a torsion pair in $\aA$. We show that under further assumptions on $\aA$ the embedding condition is equivalent to the existence of a torsion pair $(\ol{\tT}, \eE)$ in $\cC$, where $\eE$ is the torsion free part. We describe also Ext-groups of objects in $\eE$ over the quotient category $\cC$ as colimits of Ext-groups over $\aA$.

\vspace{0.3cm}
\subsection{Closed objects as a full subcategory of the quotient category}~\\

Recall that a Serre subcategory $\bB\subset \aA$ yields a  multiplicative system in $\aA$:
$$
\sS= \{f \,|\, \ker f, \textrm{coker} f \in \bB\}.
$$
Since we are interested in $\bB$-closed objects, it is more convenient to view $\sS$ as a right multiplicative system. Then morphisms in the quotient category $\cC \simeq \aA[\sS^{-1}]$ 
between two objects $A_0$ and $A_1$ in $\aA$ are equivalence classes of roofs
\begin{equation}\label{eqtn_roof}
\xymatrix{ & A' \ar[dl]_s \ar[dr]^f & \\ A_0 & & A_1}
\end{equation}
with $s\in \sS$. We denote (\ref{eqtn_roof}) by $fs^{-1}: A_0\to A_1$. The equivalence relation on roofs is analogous to (\ref{eqtn_equiv_of_roofs}).

We check that the quotient functor (\ref{eqtn_def_quot_funct}) is fully faithful when restricted to the category $\eE$  of $\bB$-closed objects  (\ref{eqtn_def_of_closed_obj}). It allows us to consider $\eE$ as a full subcategory both in $\cC$ and in $\aA$. Even more is true,

\begin{PROP}\label{prop_Hom_to_closed}
	For any $A\in \aA$, $E\in \eE$,  the quotient functor induces an isomorphism:
	\begin{equation}\label{hom-closed}
	q_{A,E} \colon \Hom_{\aA}(A, E) \simeq \Hom_{\cC}(q(A), q(E)).
	\end{equation}
	
\end{PROP}
\begin{proof}
	Let $fs^{-1} \colon A\to E$ be a morphism in $\cC$, where $f: A' \to E$ is a morphism in $\aA$ and $s\colon A'\to A$ is in $\sS$. We have an exact sequence:
	\begin{equation}\label{b'a'ab}
	0\to B'\to A' \xrightarrow{s} A\to B\to 0,
	\end{equation}
	where $B,B' \in \bB$. Since $\Hom_{\aA}(-, E)$ and $\Ext^1_{\aA}(-, E)$ vanish on $\bB$, the standard spectral sequence for $\Ext^{\bullet}_{\aA}(-,E)$, when applied to (\ref{b'a'ab}), implies $\Hom_{\aA} (A,E)\simeq_{\aA} \Hom (A', E)$, which gives a unique lifting of $f$ to a morphism in $\Hom_{\aA} (A,E)$. The lifting is independent of the choice of an element in the equivalence class of roofs.
\end{proof}

\begin{COR}\label{cor_ff_on_closed}
	Functor $q|_\eE \colon \eE\to \cC$ is fully faithful.
\end{COR}

\vspace{0.3cm}
\subsection{Inducing a torsion pair in the quotient category}\label{ssec_tor_pair_on_quotient}~\\

Given a full subcategory $\mathcal{G}\subset \aA$ we denote by $q(\gG)\subset \cC$ the \emph{essential image} of the quotient functor (\ref{eqtn_def_quot_funct}) restricted to $\gG$. In other words, $q(\gG)$ is the (strictly) full subcategory of $\cC$ with objects $C\in \cC$ for which there exists $G\in \gG$ and an isomorphism  $C\simeq q(G)$ in $\cC$.

We discuss sufficient conditions for a  torsion pair $(\tT, \fF)$ in $\aA$ to  induce a torsion pair $(q(\tT), q(\fF))$ in $\cC$.

\begin{LEM}\label{prop_torsion_pair_if_no_Hom}
	Let $(\tT, \fF)$ be a torsion pair in $\aA$. Then the subcategories $q(\tT)$, $q(\fF)$ form a torsion pair in $\cC$ if and only if $\Hom_{\cC}(q(\tT), q(\fF) ) =0$.
\end{LEM}
\begin{proof}
	We check that any object $C\in \cC$ is an extension of an object in $q(\fF)$ by an object in $q(\tT)$. Indeed,  $C$ is isomorphic to $q(A)$, for some $A\in \aA$. Let $0 \to T\to A \to F \to 0$ be the decomposition of $A$ in the torsion pair $(\tT, \fF)$. Then $0 \to q(T) \to q(A) \to q(F) \to 0$ is a short exact sequence in $\cC$ with $q(T)\in q(\tT)$ and $q(F)\in q(\fF)$.
\end{proof}

\begin{LEM}\label{lem_tor_pair_on_quot_if_T_Serre}
	Let $(\tT, \fF)$ be a torsion pair in $\aA$ with $\tT \subset \aA$ a Serre subcategory. If $\bB$ is a subcategory of $\tT$ then $(q(\tT), q(\fF))$ is a torsion pair in $\cC$.
\end{LEM}
\begin{proof}
	By Lemma \ref{prop_torsion_pair_if_no_Hom} it suffices to check that $\Hom_{\cC}(q(\tT), q(\fF)) =0$. For $T\in \tT$ and $F\in \fF$, let $f \circ s^{-1} \colon T\to F$ be a morphism in $\cC$ with $s\colon A \to T$ in $\sS$. Since $\tT \subset \aA$ is a Serre subcategory, the image of $s$ is an  object of $\tT$.  Since $A$ is an extension of the image by the kernel of $s$ and $\bB \subset \tT$, then $A$ is an object of $\tT$. Hence, $f \colon A\to F$ is the zero morphism.
\end{proof}

Let $\gG$ be a full subcategory of the category $\eE$ of $\bB$-closed objects (\ref{eqtn_def_of_closed_obj}). We say that a full subcategory $\fF\subset \aA$ satisfies the \emph{embedding condition with respect to $\gG$} if 
\begin{equation}\label{eqtn_ast}
\forall \, F\in \fF \, \exists\, G\in \gG\, \exists \, B\in \bB, \, \textrm{and a short exact sequence }0\to F\to G \to B \to 0.
\end{equation}

\begin{REM}\label{rem_uniqueness_of_emb}
The short exact sequence (\ref{eqtn_ast}) can be considered as an exact triangle 
	\begin{equation} \label{eqtn_tr}
	B[-1] \to F\to G \to B
	\end{equation} 
	in $\dD(\aA)$ or any triangulated category with a \tr e whose heart is equivalent to $\aA$. Since $\Ext^{-1}_{\dD(\aA)}(B[-1],G) \simeq 0 \simeq \Hom_{\dD(\aA)}(B[-1], G)$, the standard argument \cite[Proposition 1.1.9]{BBD} shows that if (\ref{eqtn_tr}) exists it is unique.
\end{REM}

\begin{PROP}\label{prop_emb_into_closed_gives_w_loc}
	Consider a torsion pair $(\tT, \fF)$ in $\aA$ and a full subcategory $\gG \subset \fF\cap \eE$. Assume that $\fF$ satisfies the embedding condition with respect to $\gG$. Then $(q(\tT), \gG)$ is a torsion pair in $\cC$. 
\end{PROP}
\begin{proof}
	Sequence (\ref{eqtn_ast}) implies that any object $q(F) \in q(\fF)$ is equivalent to $q(G)$, for some $G\in \eE\cap \fF$. By Proposition \ref{prop_Hom_to_closed}, the vanishing of $ \Hom_{\aA}(T, G)$, for any $T\in \tT$ and $G\in \gG$, implies the vanishing of  $\Hom_{\cC}(q(T), q(F)) = \Hom_{\cC}(q(T), q(G)) \simeq  \Hom_{\aA}(T, G)$, for any $T\in \tT$, $F\in \fF$. Hence, $(q(\tT), q(\fF))$ is a torsion pair in $\cC$ by Lemma  \ref{prop_torsion_pair_if_no_Hom}.
\end{proof}

\begin{PROP}\label{prop_cotilitng_on_quot}
			Consider a cotilting torsion pair $(\tT, \fF)$ in $\aA$ (see Appendix \ref{ssec_tros_pairs}) and a full subcategory $\gG \subset \fF\cap \eE$. Assume that $\fF$ satisfies the embedding condition with respect to $\gG$. Then the torsion pair $(q(\tT), \gG)$ in $\cC$ is cotilting.
\end{PROP}
\begin{proof}
	The existence of the torsion pair $(q(\tT), \gG)$ follows from Proposition \ref{prop_emb_into_closed_gives_w_loc}.
	
	An object $C\in \cC$ is isomorphic to $q(A)$, for some $A\in \aA$. Let $\f\colon F\to A$ be an epimorphism in $\aA$ with $F\in \fF$. Condition (\ref{eqtn_ast}) implies that $q(F) \simeq q(G)$, for some $G\in \gG$. Hence, $q(\f)\colon q(F)\to C$ is the required epimorphism with $q(F) \simeq q(G)\in q(\gG)\simeq \gG$. 
\end{proof}

\vspace{0.3cm}
\subsection{Weakly localising Serre subcategories}~\\

Given a full subcategory $\gG$ of an abelian category $\aA$ we denote by
\begin{align*}
&\gG^{\perp} = \{A\in \aA \,|\, \Hom(\gG, A) =0\},& &{}^\perp \gG= \{A \in \aA \,|\, \Hom(A, \gG) =0\}&
\end{align*} 
the full subcategories right, respectively left, orthogonal to $\gG$. Clearly, $\gG^\perp$ is closed under extensions and subobjects while ${}^\perp \gG$ is closed under extensions and quotients.

Recall the subcategory $\eE$ of $\bB$-closed objects (\ref{eqtn_def_of_closed_obj}). 
We say that $\bB \subset \aA$ is \emph{weakly localising} if $\eE$ is the torsion-free part of a torsion pair $(\ol{\tT}, \eE)$ in the quotient category $\cC$.

\begin{REM}
	If $\bB$ is a localising Serre subcategory, i.e. $q$ admits the right adjoint $r$, then $\eE \simeq \cC$ is the torsion-free part of the trivial torsion pair $(0, \eE)$ on $\cC$. Indeed, $r\colon \cC \to \eE$ is essentially surjective by \cite[Corollaire III.2]{Gabriel} and fully faithful by Corollary \ref{cor_ff_on_closed}.
\end{REM}

We show that under the assumption that ${}^\perp \eE$ is the torsion part of a torsion pair $({}^\perp \eE, \fF)$ in $\aA$, the subcategory $\bB\subset \aA$ is weakly localising if and only if $\fF$ satisfies the embedding condition (\ref{eqtn_ast}) with respect to $\eE$.

If the category $\aA$ is Noetherian the torsion pair $({}^\perp \eE, \fF)$ always exists:
\begin{LEM}\label{lem_torsion_in_noetherian}({\it cf.} \cite[Lemma 1.1.3]{Pol1})
	Let $\aA$ be a Noetherian abelian category and $\tT \subset \aA$ a subcategory closed under extensions and quotients. Then $\tT$ is the torsion part of a torsion pair in $\aA$. 
\end{LEM}

\begin{PROP}\label{prop_torsion_pair_given_by_B}
	Assume that the Serre subcategory $\bB\subset \aA$ is weakly localising and that category $\aA$ admits a torsion pair $({}^\perp \eE, \fF)$. Then $\fF$ satisfies the embedding condition with respect to $\eE$.
\end{PROP}
\begin{proof}
	Let $(\ol{\tT}, \eE)$ be the torsion pair in the quotient category $\cC$. Given $A\in \aA$, the decomposition of $q(A)$ yields an epimorphism $q(A) \to q(E)$, for some $E\in \eE$. In particular, by Proposition \ref{prop_Hom_to_closed}, we get a canonical morphism $\f_A\in \Hom_{\aA}(A,E)$. Let $K$, respectively $Q$, denote the kernel, respectively the cokernel of $\f_A$. Then $q(K)\in \ol{\tT}$ and $Q\in \bB$. Note that the condition $q(K)\in \ol{\tT}$ is equivalent to the condition $K\in {}^\perp \eE$. Indeed, by Proposition \ref{prop_Hom_to_closed} $q(K)$ is an object of $\ol{\tT}$ if and only if $0\simeq\Hom_{\cC}(q(K), \eE) \simeq \Hom_{\aA}(K, \eE)$.

	As subcategory $\fF\subset \aA$ is closed under subobjects, for any $F\in \fF$, the kernel of $\f_F$ is an object of ${}^\perp \eE \cap \fF=\{0\}$, i.e. $\f_F$ is a monomorphism. Then $0\to F\xrightarrow{\f_F} E \to Q \to 0$, with $Q\in \bB$ is the required short exact sequence. 
\end{proof}

\begin{THM}\label{thm_emb_cond}
	Assume that category $\aA$ admits a torsion pair $({}^\perp \eE, \fF)$. Then $\bB\subset \aA$ is weakly localising if and only if $\fF$ satisfies the embedding condition with respect to $\eE$.
\end{THM}
\begin{proof}
	By construction, $\eE\subset \fF$. The statement follows from Propositions \ref{prop_emb_into_closed_gives_w_loc} and \ref{prop_torsion_pair_given_by_B}.
\end{proof}

\vspace{0.3cm}
\subsection{Ext-groups between closed objects in the quotient category}~\\

A category $\gG \subset \eE$ satisfying the assumptions of Proposition \ref{prop_emb_into_closed_gives_w_loc} is closed under extensions in $\cC$. We shall present $\Ext^1_{\cC}(E_1,E_2)$, for $E_1, E_2 \in \gG$, as a colimit of appropriate Ext-groups in the category $\aA$.

Recall from Corollary \ref{cor_ff_on_closed} that the category $\eE$ of $\bB$-closed objects (\ref{eqtn_def_of_closed_obj}) is a full subcategory both in $\aA$ and in $\cC$. This allows us to consider a short exact sequence in $\cC$ with all terms in $\eE$ as a complex in $\aA$.
\begin{LEM}\label{lem_4-term_in_A}
	A short exact sequence 
	\begin{equation}\label{eqtn_ses_in_C}
	0\to E_2 \xrightarrow{f} E\xrightarrow{g} E_1 \to 0
	\end{equation}
	in $\cC$ with all terms in $\eE$ yields a 4-term exact sequence 
	\begin{equation}\label{eqtn_4-term_in_A}
	0\to E_2 \xrightarrow{f} E \xrightarrow{g} E_1 \to B \to 0
	\end{equation}
	in $\aA$, with $B\in \bB$.
	
	Conversely, any 4-term sequence (\ref{eqtn_4-term_in_A}) in $\aA$ yields a short exact sequence (\ref{eqtn_ses_in_C}) in $\cC$.
\end{LEM} 
\begin{proof}
	Since the restriction of the quotient functor $q|_{\eE} \colon \eE \to \cC$ is fully faithful by Corollary \ref{cor_ff_on_closed}, sequence (\ref{eqtn_ses_in_C}) admits a lift to a complex $E_2\xrightarrow{f} E\xrightarrow{g} E_1$ in $\aA$ with cohomology in $\bB$.

	Since there are no morphisms from $\bB$ to $\eE$, the cohomology at $E_2$ is zero. The cohomology $H$ at $E$ fits into a short exact sequence $0\to E_2 \to \ker_{\aA} g \to H \to 0$. Since $\Ext^1(\bB, E_2) =0$, this sequence splits.
Then $H$ is a subobject of $E\in \eE$, hence, as above, $H\simeq 0$.
	
	The inverse statement is clear as $q(B)=0$.
\end{proof}

For an object $E\in \eE\subset \aA$ denote by $\aA/^{\bB}E$ the category whose objects are pairs $(A, i)$ of an object $A\in \aA$ and a  monomorphism $i\colon A\xrightarrow{i} E$ in $\aA$ with the cokernel in $\bB$. Morphisms $(A',i') \to (A,i)$ in $\aA/^{\bB}E$ are $\f \in \Hom_{\aA}(A',A)$ such that $i\circ \f = i'$.

Given $E_2\in \aA$, consider a functor $\Ext^1_{\aA}(-,E_2) \colon \aA/^{\bB}E ^{\opp}\to \textrm{Ab}$ which maps $(A, i)$ to $\Ext^1_{\aA}(A,E_2)$. Given $f\in \Hom((A,i), (A',i'))$, the map $\Ext^1_{\aA}(f,E_2)\colon \Ext^1_{\aA}(A', E_2) \to \Ext^1_{\aA}(A,E_2)$ is the precomposition with $f$.

\begin{PROP}\label{prop_Ext_as_colim}
	Consider a full subcategory $\gG\subset \eE$ closed under extensions in $\cC$. 
	Then, for any $G_1, G_2\in \gG$, 
	$$
	\mathrm{Ext}^1_{\cC}(G_1,G_2) = \varinjlim_{\aA/^{\bB}G_1^{\opp}} \mathrm{Ext}^1_{\aA}(-,G_2).
	$$
\end{PROP}
\begin{proof}
	Consider $(A, i) \in \aA/^{\bB}G_1$ and $\zeta \in \Ext^1_{\aA}(A, G_2)$. The  short exact sequence $0\to G_2\xrightarrow{f}G\xrightarrow{g} A \to 0$ given by $\zeta$ combined with the short exact sequence $0\to A\xrightarrow{i} G_1 \to B\to 0$ yields a 4-term exact sequence $0\to G_2 \to G\to G_1\to B \to 0$, i.e. an element $\xi$ of $\Ext^1_{\cC}(G_1,G_2)$  (see Lemma \ref{lem_4-term_in_A}). Denote by $r_A\colon \Ext^1_{\aA}(A, G_2) \to \Ext^1_{\cC}(G_1,G_2)$ the morphism which sends $\zeta$ to $\xi$. It is straightforward to check that $r_A$ is well-defined. Indeed, if $s\colon A\to G$ is a section of $g$, then $q(s)\circ q(i)^{-1}$ is a section of $q(g)$. 
	
	Let now $\f$ be an element of $\Hom((A', i'), (A,i))$. The snake lemma for diagram
	\[
	\xymatrix{0 \ar[r] & A \ar[r]^i & G_1 \ar[r] & B \ar[r] & 0\\
		0 \ar[r] & A' \ar[r] \ar[u]^{\f} & A'\ar[r] \ar[u]^{i'} & 0 \ar[r] \ar[u] & 0}
	\]
	implies that $\f$ is a monomorphism with cokernel in $\bB$. The pull back of an extension $0 \to G_2 \to G\to A \to 0$ along $\f$ fits into a commutative diagram
	\begin{equation}\label{eqtn_4_1}
	\xymatrix{0 \ar[r] & G_2 \ar[r] & G\ar[r] & A \ar[r]^i & G_1\\
	0 \ar[r] & G_2\ar[r] \ar[u]^{\Id} & G'\ar[r] \ar[u]^{\psi} & A'\ar[r]^{i'} \ar[u]^{\f}& G_1 \ar[u]^{\Id}}
	\end{equation}
	Since $\textrm{coker} \psi \simeq \textrm{coker} \f$ is an object of $\bB$, $q(\psi)$ is an isomorphism in $\cC$, i.e. (\ref{eqtn_4_1}) yields a commutative diagram in $\cC$:
	\[
	\xymatrix{0 \ar[r] & G_2\ar[r]& G\ar[r] & G_1 \ar[r] & 0\\
	0 \ar[r] & G_2\ar[r]\ar[u]^{\Id}& G' \ar[r] \ar[u]^{q(\psi)} & G_1 \ar[r]\ar[u]^{\Id} & 0}
	\]
	 It follows that the diagram
	\[
	\xymatrix{\Ext^1_{\aA}(A,G_2) \ar[rr]^{\Ext^1_{\aA}(\f, G_2)} \ar[dr]_{r_A} && \Ext^1_{\aA}(A',G_2) \ar[dl]^{r_{A'}}\\
	& \Ext^1_{\cC}(G_1,G_2)& }
	\]
	commutes, hence we get a morphism $\Upsilon\colon \varinjlim_{\aA/^{\bB}G_1^{\opp}} \Ext^1_{\aA}(-,G_2)\to \Ext^1_{\cC}(G_1,G_2)$.
	
	As $\gG\subset \cC$ is closed under extensions, Lemma \ref{lem_4-term_in_A} implies that $\Upsilon$ is surjective.
	
	An element $\zeta$ in the kernel of $\Upsilon$ is a class of a short exact sequence $0\to G_2\xrightarrow{f} G \xrightarrow{g} A \to0$, for some $(A,i)\in \aA/^{\bB}G_1$. The splitting of $r_A(\zeta)$ is a morphism $s\in \Hom_{\cC}(G_1,G)$. By Corollary \ref{cor_ff_on_closed}, $s$ can be considered as a morphism in $\aA$. Since $s$ is the splitting of $r_A(\zeta)$, $i\circ g \circ s = \Id_{G_1}$. Hence, $i\circ g \circ s \circ  i = i$. As $i\colon A\to G_1$ is a monomorphism, the last equality implies that $g\circ s\circ i = \Id_{A}$. It follows that $s\circ i$ is a section of $g$, i.e. $\zeta =0$. 
\end{proof}

\begin{COR}\label{cor_can_exact_str}
	Consider a weakly localising Serre subcategory $\bB$ in an abelian category $\aA$. Then the category $\eE$ of $\bB$-closed objects is quasi-abelian and conflations in the canonical exact structure on $\eE$ are 4-term exact sequences $0\to E_1 \to E_2\to E_3 \to B \to 0$ in $\aA$ with $E_i \in \eE$ and $B\in \bB$. Moreover, $\Ext^1_{(\eE, \sS_{\textrm{can}})}(E_1, E_2) =\varinjlim_{\aA/^{\bB}E_1^{\opp}} \Ext^1_{\aA}(-,E_2)$. 
\end{COR}
\begin{proof}
	Category $\eE$ is the torsion-free part of a torsion pair in the quotient category $\aA/\bB$.  It follows (see Proposition \ref{prop_t_tf_is_q-a})  that $\eE$ is quasi-abelian and the canonical exact structure $(\eE, \sS_{\textrm{can}})$ is the one induced by the inclusion $\eE\subset \aA/\bB$. Since the torsion-free part of a torsion pair is closed under extensions, the description of conflations in $\sS_{\textrm{can}}$ follows from Lemma \ref{lem_4-term_in_A} and Proposition \ref{prop_Ext_as_colim}.
\end{proof}

	\section{Reflexive sheaves on a surface and their right abelian envelope}\label{sec_ref_sheaves}

		Here we discuss the category $\Rref(X)$ of reflexive sheaves on a  normal surface $X$. We prove that it is quasi-abelian and its right abelian envelope is the quotient of $\Coh(X)$ by the subcategory of Artinian sheaves. We show that the cotiliting torsion pair in $\aA_r(\Rref(X))$ (see Proposition \ref{prop_cotilt_is_quasi-ab}) is canonical. We argue that $H^0(X, \oO_X)$ is the center of $\Rref(X)$ and 
 discuss the functor $\Rref(X) \to \Rref(U)$, for an open subset $U\subset X$.

	\vspace{0.3cm}
	\subsection{The right abelian envelope of $\Rref(X)$ }~\\
	
	Denote by $\Coh_{\leq 0 }(X) \subset \Coh(X)$ the Serre subcategory of Artinian sheaves and by $\Coh^{\geq 1}(X) = \Coh(X)/\Coh_{\leq 0}(X)$ the quotient. More generally, 	given a Noetherian scheme $Z$, denote by $\Coh_{\leq k}(Z)\subset\Coh(Z)$ the full subcategory of sheaves supported in dimension less than or equal to $k$ any by $\Coh^{\geq k+1}(Z) = \Coh(Z)/\Coh_{\leq k}(Z)$ the quotient.
	
	We show that $\Coh^{\geq 1}(X)$ is the right abelian envelope of $\Rref(X)$ endowed with its canonical exact structure.

	 It is well-known, \emph{cf.} \cite[Proposition 1.5.3]{Yoshino}, that the category $\Rref(X)$ of reflexive sheaves on $X$ is the the category of $\Coh_{\leq 0}(X)$--closed objects:
 	\begin{equation}\label{eqtn_locally_closed} 
	\Rref(X)=\{E\in \Coh(X) \,|\, \Hom_X(\Coh_{\leq 0}(X), E) =0 = \Ext^1_X(\Coh_{\leq 0}(X), E)\}.
	\end{equation} 
	
	Let $(\Coh_{\leq 1}(X), \fF(X))$ be the torsion pair in $ \Coh(X)$ of torsion and torsion-free sheaves. 
	\begin{LEM}\label{lem_emb_cond_for_Noeth_sch}
		Category $\fF(X)$ satisfies the embedding condition with respect to $\Rref(X)$.
	\end{LEM}
\begin{proof}
	For $F\in \fF(X)$, the canonical morphism to its double dual is an embedding. Since $X$ is regular in codimension one, the cotorsion $F^{\vee\vee}/F \in \Coh_{\leq 0}(X)$ is an Artinian sheaf.
\end{proof}

	\begin{THM}\label{thm_tor_pair_for_eqidim_scheme}
		Let $X$ be a normal surface.
		Then 
		$\Coh^{\geq 1}(X)$ admits a cotilting torsion pair $(\Coh_{\leq 1}(X)/\Coh_{\leq 0}(X), \Rref(X))$. In particular, $\Rref(X)$ is quasi-abelian.
	\end{THM}	
	\begin{proof}
		By Lemma \ref{lem_emb_cond_for_Noeth_sch}, $\fF(X)$ satisfies the embedding condition with respect to $\Rref(X)$. Since $\Rref(X)$ is a full subcategory of $\fF(X)$, the existence of the torsion pair on $\Coh(X)/\Coh_{\leq 0}(X)$ follows from Proposition \ref{prop_emb_into_closed_gives_w_loc}.

		 By \cite[Theorem 2.1]{SchrVezz}, any coherent sheaf on $X$ is a quotient of a locally free sheaf. In particular, the pair $(\Coh_{\leq 1}(X), \fF(X))$ is cotilting. Hence, so is the induced torsion pair 
		 on the quotient, see Proposition \ref{prop_cotilitng_on_quot}.
		\end{proof}
	
	In view of Corollary \ref{cor_can_exact_str} category $\Rref(X)$ admits the \emph{canonical exact structure} in which conflations correspond to 4-term exact sequences in $\Coh(X)$:
	$$
	0 \to E_1 \to E_2 \to E_3 \to T \to 0
	$$
	with $E_i \in \Rref(X)$ and $T\in \Coh_{\leq 0}(X)$.
	\begin{COR}\label{cor_env_of_E_n-2}
		Let $X$ be a normal surface. Then $\Coh^{\geq 1}(X)$ is the right abelian envelope of $\Rref(X)$ endowed with the canonical exact structure.
	\end{COR}
	\begin{proof}
		Follows from Proposition \ref{prop_cotilt_is_quasi-ab} and Theorem \ref{thm_tor_pair_for_eqidim_scheme}.
	\end{proof}
For a Noetherian scheme $X$ denote by
 $\Pic_{\geq k}(X)$ the group of isomorphism classes in $\Coh^{\geq k}(X)$ of $j_*\lL$ where $j \colon U \to X$ is an open embedding with complement of codimension $k-1$ and an element $\lL$ of the group $\textrm{Pic}(U)$ of isomorphism classes of invertible sheaves on $U$. Any $\lL\in \Pic_{\geq k}(X)$ gives an autoequivalence $T_{\lL}$ of $\Coh^{\geq k}(X)$ defined as the composition of the equivalence $j^*\colon \Coh^{\geq k}(X)\xrightarrow{\simeq} \Coh^{\geq k}(U)$ (see Theorem \ref{thm_cal_pir} below) with the twist by $\lL$ and $(j^*)^{-1}$.
\begin{THM}\cite[Theorem 4.1]{CalPir}\label{thm_CP2}
	Let $X$ be a Noetherian scheme of finite type over an algebra of finite type over a field. 
	Given an equivalence $\Phi \colon \Coh^{\geq k}(X) \xrightarrow{\simeq} \Coh^{\geq k}(X)$ there exists $\lL\in \Pic_{\geq k}(X)$ and a birational morphisms $\alpha \colon X\dashrightarrow X$ which is biregular on open subsets $U,V\subset X$ with complement of codimension $k-1$ such that $\Phi \simeq T_{\lL} \circ \alpha^*$.
\end{THM}
\begin{REM}
	In \cite{CalPir} the authors did not say clear	ly that $\alpha$ should be biregular on an open subset with complement of codimension $k-1$ but only defined on it. Cremona transformation and its inverse show that it is not enough. Indeed, they are defined on the complement to three points in $\mathbb{P}^2$ but are biregular on the complement to three lines.
	Nevertheless, the argument of the proof of \cite[Theorem 4.1]{CalPir} proves Theorem \ref{thm_CP2} as formulated above.
\end{REM}

Now we are ready to describe the group $\Auteq(\Rref(X))$ of autoequivalences of $\Rref(X)$. 
For a normal surface $X$ denote by $X^o$ its non-singular part, $X^o = X\setminus X_{\textrm{sing}}$. Clearly, $X^o \to X$ is a morphism in $\dD$. 
\begin{THM}\label{thm_autoeq}
	Let $X$ be a normal surface. Then $\Auteq(\Rref(X)) \simeq \Aut(\wh{X}) \ltimes \Pic(X^o)$.
\end{THM}
\begin{proof}
	Since $\Coh^{\geq 1}(X)$ is the right abelian envelope of $\Rref(X)$ (see Corollary \ref{cor_env_of_E_n-2}) any autoequivalence of $\Rref(X)$ induces an autoequivalence of $\Coh^{\geq 1}(X)$. 
	
	According to Theorem \ref{thm_CP2}, any autoequivalence of $\Coh^{\geq 1}(X)$ is given by $\lL \in \Pic_{\geq 1}(X)$ and a birational morphism $\alpha \colon X\dashrightarrow X$ which is a biregular isomorphism $\alpha \colon U \to V$ of open sets $U, V \subset X$ such that $i\colon U \to X$, $j\colon V\to X$ are morphisms in $\dD$.  
	Proposition \ref{prop_left_ms} allows us to extend $j_X \circ i, j_X \circ j \circ \alpha \colon U \to \wh{X}$ to a commutative square
		\[
	\xymatrix{ U \ar[rr]^{j_X \circ j\circ \alpha} \ar[d]^{j_X\circ i}  && \wh{X} \ar[d]^{\wt{i}}  \\
		\wh{X} \ar[rr]_{\wt{j}} & & Z} 
	\]
	with $\wt{i}$ and $\wt{j}$ in $\dD$. Since $\wt{X}$ is codim-2-saturated, $\wt{i}$ and $\wt{j}$ are biregular isomorpisms. Then $\wt{i}^{-1} \circ \wt{j}$ is a biregular isomorphism of $\wt{X}$ which extends $\alpha$. 
	
	On the other hand, any $\lL\in \Pic_{\geq 1}(X)$, i.e. $\lL\in \Pic(U)$ for an open $U\subset X$ with complement of codimension two, yields a line bundle on $U \cap X^o$. As $U\cap X^o \to X^o$ is an open embedding of non-singular normal surfaces with complements of codimension two, $\lL|_{U\cap X^o}$ admits a unique extension to $\wt{\lL} \in \Pic(X^o)$. Hence, $\Pic_{\geq 1}(X) \simeq \Pic(X^o)$.
	
	We conclude by Theorem \ref{thm_CP2} that $\Auteq(\Coh^{\geq 1}(X)) \simeq \Aut(\wh{X}) \ltimes \Pic(X^o)$. Any element $T_{\lL} \circ \alpha^*$ of this group restricts to an autoequivalence of $\Rref(X)$. The statement follows.
\end{proof}
\begin{EXM}\label{exm_cyclic}
	Consider  a field $k$ of characteristic not dividing $n$, a linear action of $\mathbb{Z}_n$ on $k^2= \Spec S$ and an affine variety $X = k^2/\mathbb{Z}_n$ with an isolated cyclic singularity at $x_0\in X$. By Auslander Theorem \cite[Proposition 2.2]{Auslander}, the category $\Rref(X)$ is equivalent to the category of projective $S[\mathbb{Z}_n]$-modules, hence to the category 
	$\Rref_{\mathbb{Z}_n}(k^2)$ of $\mathbb{Z}_n$-equivariant  vector bundles on $k^2$. Indecomposable objects in $\Rref_{\mathbb{Z}_n}(k^2)$ are parametrised by characters $\chi$ of $\mathbb{Z}_n$, they are of the form $k_{\chi}\otimes \oO_{k^2}$. The corresponding reflexive sheaf on $X$ is a line bundle when restricted to $X^o = X\setminus \{x_0\}$, hence defines an element of $\Pic(X^o)$. Since the product of characters of $\mathbb{Z}_n$ corresponds to product of elements in $\Pic(X^o)$, the latter group is isomorphic to $\mathbb{Z}_n$. It follows that the group $\Pic(X^o)$ cyclically permutes indecomposable objects in  $\Rref(X)$. For $k= \mathbb{C}$ the group $\Aut(X)$ is described in \cite[Theorem 4.2]{ArzZai}. 
\end{EXM}

	\vspace{0.3cm}
	\subsection{The canonical torsion pair on $\aA_r(\Rref(X))$}~\\
	
	Recall that an object of an abelian category is \emph{simple} if it has no non-trivial subobjects. An object is of \emph{finite length} if it admits a finite filtration with simple graded factors.

	\begin{PROP}\label{prop_Cal-Pir}\cite[Proposition 1.4]{CalPir}
		For a scheme $Z$ and $k\leq \dim Z$, 
		the subcategory $\Coh_{\leq k-1}(Z)/\Coh_{\leq k-2}(Z) \subset \Coh^{\geq k-1}(Z)$ is the Serre subcategory of objects of finite length in  $\Coh^{\geq k-1}(Z)$. 
	\end{PROP}

Denote by $X^1$ the set of points in the scheme $X$ corresponding to prime divisors. Given $D\in X^1$, we denote by $\Coh_D^{\geq 1}(X)$ the quotient of the category $\Coh_D(X)$ of sheaves supported on the closure $\ol{D}$ of $D$ by the Serre subcategory $\Coh_D(X)\cap \Coh_{\leq 0}(X)$. 

By \cite[Lemma 2.5]{CalPir} the quotient of $\Coh^{\geq 1}(X)$ by the minimal Serre subcategory containing $\oO_Z$, for all irreducible closed subsets $Z\subset X$ such that $D\notin \textrm{Supp }\oO_Z$, is equivalent to the category of finitely generated modules over  the local ring $\oO_{X,D}$ of $X$ at the point $D$. This equivalence restricts to an equivalence of $\Coh^{\geq 1}_D(X)$ with the category $\textrm{mod--}\oO_{X,D}$ of finite length modules over $\oO_{X,D}$. In particular,  $\Coh_D^{\geq 1}(X)$ is a finite length category with one simple object $\oO_D$.

\begin{PROP}\label{prop_canon_torsion_pair_Coh_1}
	Consider a normal surface $X$. Category $\Coh^{\geq 1}(X)$ admits a
	cotilting torsion pair $(\bigoplus_{D\in X^1} \textrm{mod--}\oO_{X,D}, \Rref(X))$, where $\bigoplus_{D\in X^1}\textrm{mod--}\oO_{X,D}$ is the Serre subcategory of objects of finite length.
\end{PROP}
\begin{proof}
	In view of Theorem \ref{thm_tor_pair_for_eqidim_scheme}, it suffices to show that $\Coh_{\leq 1}(X)/\Coh_{\leq 0}(X) \simeq \bigoplus_{D\in X^1} \textrm{mod--}\oO_{X,D}$. 
	By Proposition \ref{prop_Cal-Pir}, $\Coh_{\leq 1}(X)/\Coh_{\leq 0}(X) $ is the subcategory of finite length objects in $\Coh^{\geq 1}(X)$.
	By\cite[Lemma 1.6]{CalPir}, any simple object in $\Coh(X)/\Coh_{\leq 0}(X)$ is isomorphic to $\oO_D$, for a prime divisor $D$. 
	Note that, for different prime divisors $D_1, D_2 \in X^1$,  we have $\Ext^1_{\Coh^{\geq 1}(X)}(\oO_{D_1}, \oO_{D_2})=0$. Indeed,  as $\Coh^{\geq 1}(X) \simeq \Coh^{\geq 1}(U)$ for an open subset $U\subset X$ complement to the intersection points of $D_1$ and $D_2$,
	we can assume that the support of $D_1$ and $D_2$ do not intersect. Hence, there are no nontrivial extensions of $\oO_{D_1}$ by $\oO_{D_2}$. The statement follows.
\end{proof}

	\begin{PROP}\label{prop_Ref_from_Coh}
	Given a normal surface $X$, the category $\Rref(X)$ is the full subcategory of $\Coh^{\geq 1}(X)$ right orthogonal to the full subcategory of simple objects.
	\end{PROP} 
	\begin{proof}
		Follows from Proposition \ref{prop_canon_torsion_pair_Coh_1}.
	\end{proof}

	\vspace{0.3cm}
	\subsection{ $H^0(X, \oO_X)$ as the center of the category $\Rref(X)$}~\\

	We show that the affinisation of a normal surface $X$ can be recovered from $\Rref(X)$.
	
	For a category $\eE$ denote by $Z(\eE) :=\End(\Id_{\eE})$ its \emph{center}.
	
	\begin{LEM}\label{lem_center_of_envelope}
		Let $\eE$ be an additive category. For any exact structure $\sS$ on $\eE$ such that $i_R\colon \eE \to \aA_r(\eE,\sS)$ is fully faithful, the center $Z(\eE)$ of $\eE$ equals the center $Z(\aA_r(\eE, \sS))$ of its right abelian envelope.
	\end{LEM}	
	\begin{proof}
		Since $i_R$ is fully faithful, the endomorphism ring of $\Id_{\eE}$ equals the endomorphism ring of $i_R$ in $\Rex(\eE, \aA_r(\eE, \sS))$. The statement follows from the fact that the equivalence $(-)\circ i_R \colon \Rex(\aA_r(\eE, \sS), \aA_r(\eE, \sS)) \xrightarrow{\simeq} \Rex(\eE, \aA_r(\eE, \sS))$ maps $\Id_{\aA_r(\eE, \sS)}$ to $i_R$.
	\end{proof}
	
	The subcategory $\Rref(X)\subset \Coh(X)$ is closed under extensions. Hence, $\Rref(X)$ admits the induced exact structure in which conflations are short exact sequences in $\Coh(X)$ with all terms in $\Rref(X)$. We shall refer to this exact structure as the \emph{geometric} one.
\begin{EXM}
	In Example \ref{exm_cyclic} object $\oO_X\in \Rref(X)$ is the only indecomposable projective object in the geometric exact structure. Since the group $\Pic(X^o)$ permutes indecomposable objects in $\Rref(X)$, the corresponding autoequivalences of $\Rref(X)$ (see Theorem \ref{thm_autoeq}) do not preserve the geometric exact structure. 
\end{EXM}	
	\begin{LEM}\label{lem_coh_as_env_of_E_0}
		Consider a normal surface $X$. Then $\Coh(X)$ is the right abelian envelope of $\Rref(X)$ endowed with the geometric exact structure.
	\end{LEM}
	\begin{proof}
		The property (\ref{eqtn_locally_closed}) of $\Rref(X)$ implies that $\Rref(X)\subset \Coh(X)$ is closed under extensions and kernels of epimorphisms. Clearly, any locally free sheaf on $X$ is an object of $\Rref(X)$. By \cite[Theorem 2.1]{SchrVezz}, any coherent sheaf on $X$ is a quotient of a locally free sheaf. Hence, by Proposition  \ref{prop_r_ab_env}, $\Coh(X) $ is the right abelian envelope of its full subcategory $\Rref(X)$ endowed with the induced exact structure.
	\end{proof}

	\begin{LEM}\label{lem_center_of_Coh}\cite[Remark 4.5]{Rou3}
		Consider a scheme $X$. Then $H^0(X, \oO_X)$ is the center of $\Coh(X)$.
	\end{LEM}
	
	\begin{PROP}\label{prop_cent_of_E_0}
		Consider a normal surface $X$. Then $Z(\Rref(X)) = H^0(X, \oO_X)$.
	\end{PROP}
	\begin{proof}
		By Lemmas \ref{lem_center_of_envelope}, \ref{lem_coh_as_env_of_E_0} and \ref{lem_center_of_Coh}, $Z(\Rref(X)) = Z(\aA_r(\Rref(X))) = Z(\Coh(X)) = H^0(X, \oO_X)$.
	\end{proof}
	
		\vspace{0.3cm}
	\subsection{$\Rref(X)$ and isomorphism of surfaces outside of dimension zero}~\\
	
	We show that the categories of reflexive sheaves distinguish surfaces up to closed points.
	
	Following \cite{CalPir} we say that schemes $Z_1, Z_2$ are \emph{isomorphic outside of dimension} $k-1$ if there exist isomorphic open subsets $U_1\subset Z_1$, $U_2\subset Z_2$ such that any point $p\in Z_1$ of dimension greater than or equal to $k$ is contained in $U_1$ and, analogously, any point $q\in Z_2$ of dimension greater than	or equal to $k$ is contained in $U_2$. 
	Clearly, normal surfaces $Z_1, Z_2$ are isomorphic outside of dimension zero if and only if they lie in the  same connected component of the category $\mathscr{D}$.

	\begin{THM}\cite[Theorem 3.7]{CalPir}\label{thm_cal_pir}
		Let $Z_1$, $Z_2$ be schemes of finite type over an algebra $R$ of finite type over a field. Then $Z_1$ and $Z_2$ are isomorphic outside of dimension $k-1$ if and only if $\Coh^{\geq k} \simeq \Coh^{\geq k}(Z_2)$.
	\end{THM}
	
	\begin{PROP}\label{prop_Ref_emb_codim_2}
		Given a normal surface $X$ and an open $U\subset X$ with complement of codimension two, the restriction functor $\Rref(X)\to \Rref(U)$ is an equivalence.
	\end{PROP}
\begin{proof}
	This follows from the fact that reflexive sheaves on normal schemes are normal in the sense of Barth, \cite[Proposition 1.6]{Har_sta}.
\end{proof}
	
	\begin{THM}\label{thm_iso_iff_E_n-2}
		Let $X$ and $Y$ be normal surfaces.
		Then $X$ and $Y$ lie in the same connected component of the category $\mathscr{D}$
		if and only if there exists an equivalence $\Rref(X)\simeq \Rref(Y)$ of  additive categories.
	\end{THM} 
	\begin{proof}
		By Corollary \ref{cor_env_of_E_n-2}, 
		an equivalence $\Rref(X) \simeq \Rref(Y)$ implies an equivalence $\Coh^{\geq 1}(X) \simeq \Coh^{\geq 1}(Y)$ of the right abelian envelopes, i.e. an isomorphism of $X$ and $Y$ outside of dimension zero, see Theorem \ref{thm_cal_pir}.
				
		By Proposition \ref{prop_Ref_emb_codim_2}, if $X$ and $Y$ 
		lie in the same connected component of $\mathscr{D}$, then
		$\Rref(X) \simeq \Rref(Y)$.
	\end{proof}
	
\begin{COR}\label{cor_iso_iff_E_n-2}
	A codim-2-saturated normal surface $X$ can be uniquely reconstructed from the additive category $\Rref(X)$ of reflexive sheaves.
	\end{COR} 
	
\begin{REM}
	Since, by Corollary \ref{cor_env_of_E_n-2}, the right abelian envelope of $\Rref(X)$ is $\Coh^{\geq 1}(X)$ and it's derived category is the quotient of $\dD^b(\Coh(X))$ by the full subcategory of complexes whose support of cohomology has codimension two \cite[Proposition 3.6]{MeiPar}, Corollary \ref{cor_iso_iff_E_n-2} for the case of smooth projective varieties follows from \cite[Theorem 5.2]{MeiPar}.
\end{REM} 	
	
	\vspace{0.3cm}
	\subsection{Reflexive sheaves under localisation}\label{ssec_restr_Ref_to_open}~\\
	
	For a normal surface $X$, we discuss the restriction of  $\Rref(X)$ to an open subset $U \subset X$. Let $D$ be the divisorial part of $X\setminus U$. If $D= \emptyset$, then $\Rref(X) \simeq \Rref(U)$ by Proposition \ref{prop_Ref_emb_codim_2}.
	If $D\neq \emptyset$, let $D = \bigcup_{i\in I} D_i$ be its presentation as a union of prime divisors $D_i \in X^1$. Consider subcategory $\sS_D = \bigoplus_I \Coh^{\geq 1}_{D_i}(X)$ in $\Coh^{\geq 1}(X)$. 
Category $\sS_D$ is the subcategory of sheaves with support in $D$, hence it is a Serre subcategory in $\Coh^{\geq 1}(X)$.

	\begin{LEM}\label{lem_quotient_by_S_D}
		The quotient $\Coh^{\geq 1}(X)/\sS_D$ is equivalent to $\Coh^{\geq 1}(U)$.
	\end{LEM}
	\begin{proof}
		Let $Z$ be the Artinian part of $X\setminus U$. By Theorem \ref{thm_cal_pir}, $\Coh^{\geq 1}(X) \simeq \Coh^{\geq 1}(X\setminus Z)$, hence without loss of generality we can assume that $Z = \emptyset$.  Then clearly  $ \Coh^{\geq 1}(U) = (\Coh(X)/\Coh_D(X))/ \Coh_{\leq 0}(U)\simeq \Coh^{\geq 1}(X)/\Coh_D^{\geq 1}(X) = \Coh^{\geq 1}(X)/\sS_D$ (see \cite[Corollary I.9.4.8]{EGAI} for the isomorphism $\Coh(X)/\Coh_{D}(X) \simeq \Coh(U)$).
	\end{proof}
	
	\begin{PROP}\label{prop_Ref_U_from_Ref_X}
		Let $U$ be an open subset of a normal surface  $X$ and $D$ the divisorial part of $X\setminus U$. Then $\Rref(U)$ is the essential image of $\Rref(X) \subset \Coh^{\geq 1}(X)$ under the quotient functor $q\colon\Coh^{\geq 1}(X) \to \Coh^{\geq 1}(X)/\sS_D$.
	\end{PROP}
	\begin{proof}
		As $\bigoplus_{D\in X^1} \Coh^{\geq 1}_D(X)$ is a Serre subcategory in $\Coh^{\geq 1}(X)$ and $\sS_D$ is its Serre subcategory, the torsion pair $(\bigoplus_{D\in X^1} \Coh^{\geq 1}_D(X), \Rref(X))$ in $\Coh^{\geq 1}(X)$ descends to the quotient $\Coh^{\geq 1}(U)\simeq \Coh^{\geq 1}(X)/\sS_D$, see Lemmas \ref{lem_tor_pair_on_quot_if_T_Serre} and \ref{lem_quotient_by_S_D}. 
		For any prime divisor $E \subset X$ whose generic point lies in $U$,  $\Coh_E^{\geq 1}(X) \simeq \Coh_E^{\geq 1}(U)$ as both are equivalent to $\textrm{mod--}\oO_{X,E}$. Hence, the torsion part of the torsion pair on $\Coh^{\geq 1}(U)$ induced from $\Coh^{\geq 1}(X)$ is equivalent the torsion part of the canonical torsion pair of Proposition \ref{prop_canon_torsion_pair_Coh_1}. Hence, the torsion-free parts $q(\Rref(X))$ and $\Rref(U)$ are also equivalent.
	\end{proof}

	\section{Reconstruction of $\Coh(\wh{X})$ from $\Rref(X)$}\label{sec_reconstr}

 	We give a categorical characterisation of $\Rref(U)$, for a quasi-affine normal surface $U$.  This allows us to construct the category $\Coh(\wh{X})$ of coherent sheaves on the codim-2-saturated model of  a normal surface $X$ from the (additive) category $\Rref(X)$. By Appendix, we can further reconstruct $\wh{X}$ from $\Coh(\wh{X})$.
 	
	\vspace{0.3cm}
	\subsection{CM affine surface is quasi-affine}~\\

	Given a line bundle $\lL$ on $X$, the ring $\End_{\Coh(X)}(\lL)$ equals $H^0(X, \oO_X)$. Since $\lL$ is an object of $\Rref(X)$, Corollary \ref{cor_ff_on_closed}  implies that 
	\begin{equation}\label{eqtn_center_from_line_bundle}
	H^0(X,\oO_X) = \End_{\Coh^{\geq 1}(X)}(\lL).
	\end{equation} 
	Since an open embedding $U\to X$ with complement of codimension two induces an equivalence $\Coh^{\geq 1}(X) \simeq \Coh^{\geq 1}(U)$, equality (\ref{eqtn_center_from_line_bundle}) holds for any line bundle $\lL\in \Coh(U)$, i.e. for any line bundle $\lL\in \Coh(U)$, $\End_{\Coh^{\geq 1}(X)}(\lL) = Z(\Rref(X))$, see Proposition \ref{prop_cent_of_E_0}.
	
	Denote by 
	$$
	q\colon \Coh(X) \to \Coh^{\geq 1}(X)
	$$
	the quotient functor.
	
	Object $\oO_X\in \Coh^{\geq 1}(X)$ can be recovered from $\Coh^{\geq 1}(X) \simeq \aA_r(\Rref(X))$ up to an auto-equivalence of the category:
	\begin{PROP}\cite[Proposition 2.7, Theorem 4.1]{CalPir}\label{prop_Cal_Pir}
		An object $E\in \Coh^{\geq 1}(X)$ is of the form $q(j_*\lL)$, for  an open  embedding $j\colon U\subset X$ with complement of codimension two and a line bundle $\lL\in \Coh(U)$, if and only if
		\begin{enumerate}
			\item for any simple object $P\in \Coh^{\geq 1}(X)$, $\Hom_{\Coh^{\geq 1}(X)}(E, P)$ is a one-dimensional $\End_{\Coh^{\geq 1}(X)}(P)$-vector space,
			\item $E$ is locally maximal with this property, i.e. for any $E'\in \Coh^{\geq 1}(X)$ satisfying $(1)$ any surjection $E' \to E$ is an isomorphism.
		\end{enumerate}
		For any such $E$, there exists an auto-equivalence $\Phi$ of $\Coh^{\geq 1}(X)$ such that $\Phi(E)\simeq \oO_X$.
	\end{PROP}

	Since the center $Z(\eE)$ of an additive category is a commutative ring, the category $\textrm{CM}(Z(\eE))$ of maximal Cohen-Macaulay modules over it is well-defined.
	
	We say that an additive category $\eE$ is \emph{CM affine} if there exists an object $\lL \in \eE$ with $\End_{\eE}(\lL) = Z(\eE)$ such that $\Hom_{\eE}(\lL,-)\colon \eE\to \textrm{Mod--}Z(\eE)$ induces an equivalence $\eE \simeq \textrm{CM}(Z(\eE))$.
	
	We say that a normal surface $X$ is \emph{CM affine} if the category $\Rref(X)$ is.
	
		\begin{PROP}\label{prop_affin_of_CM_affine}
		If $X$ is a  CM affine normal surface then $X$ is quasi-affine.
	\end{PROP}
	\begin{proof}
		By \cite[Corollary 6.3]{Schroer}, $X^{\aff} = \Spec Z(\Rref(X))$ (see Proposition \ref{prop_cent_of_E_0}) is a normal scheme. If dimension of $X^{\aff}$ is less than two, then $X^{\aff}$ is regular and $\textrm{CM}(X^{\aff})$ is the category of torsion-free sheaves on $X^{\aff}$ by \cite[Proposition 1.5.1]{Yoshino}. In particular, if $X^{\aff}$ is zero dimensional, then $\textrm{CM}(X^{\aff})$ is abelian, hence isomorphic to its right abelian envelope. As $\Rref(X) \subsetneq \aA_r(\Rref(X))$, by Proposition \ref{prop_canon_torsion_pair_Coh_1}, we conclude that $\dim X^{\aff} \geq 1$. If $X^{\aff}$ is a (regular) curve, then $\Coh(X^{\aff})$ admits a cotilting torsion pair $(\Coh_{\leq 0}(X^{\aff}), \textrm{CM}(X^{\aff}))$ and $\aA_r(\textrm{CM}(X^{\aff})) \simeq \Coh(X^{\aff})$, by Proposition \ref{prop_cotilt_is_quasi-ab}. In particular, simple objects in $\aA_r(\textrm{CM}(X^{\aff}))$, i.e. structure sheaves of closed points, have endomorphisms algebras of finite dimension over $k$. As this is not the case for endomorphisms $\oO_{X,D}/\mathfrak{m}_{X,D}$ of simple objects $\oO_{D}$ in $\aA_r(\Rref(X)) = \Coh_{\geq 1}(X)$, we conclude that $X^{\aff}$ is a normal surface. In particular, $\textrm{CM}(X^{\aff}) \simeq\Rref(X^{\aff})$ \cite[Proposition 1.5.3]{Yoshino}.
		
		Let $\lL \in \Rref(X)$ be such that $\Psi(-) = \Hom(\lL,-)\colon \Rref(X) \to \Rref(X^{\aff})$ is an equivalence. It induces an equivalence $\ol{\Psi}\colon \Coh^{\geq 1}(X) \xrightarrow{\simeq}\Coh^{\geq 1}(X^{\aff})$ of the right abelian envelopes, by Corollary \ref{cor_env_of_E_n-2}. As the equivalence $\ol{\Psi}$ maps $ \lL$ to $\oO_{X^{\aff}}$, and $\oO_{X^{\aff}}$ satisfies conditions $(1)$ and $(2)$ of Proposition \ref{prop_Cal_Pir}, so does $\lL$. Hence, by the mentioned proposition, there exists an autoequivalence $\Phi$ of $\Coh_{\geq 1}(X)$ such that $\Phi(\lL) = \oO_X$. Since any autoequivalence maps simple objects to simple objects, $\Phi$ preserves the canonical torsion pair $(\bigoplus_{D\in X^1} \Coh_{D}^{\geq 1}(X), \Rref(X))$ of Proposition \ref{prop_canon_torsion_pair_Coh_1}. In particular, $\Phi$ restricts to an autoequivalence of $\Rref(X)$. Hence $\oO_X$ satisfies the same property as $\lL$, and we can assume that $\lL= \oO_X$.			
		
		 We check that the equivalence $\ol{\Psi} \colon \Coh^{\geq 1}(X) \xrightarrow{\simeq} \Coh^{\geq 1}(X^{\aff})$ is given by $\Hom_{\Coh^{\geq 1}(X)}(\oO_X,-)$. To this end, consider the geometric exact structure $\sS_g$ on $\Rref(X^{\aff})$ induced by the inclusion $\Rref(X^{\aff}) \subset \Coh(X^{\aff})$. Then $\aA_r(\Rref(X^{\aff}), \sS_g) \simeq \Coh(X^{\aff})$ by Lemma \ref{lem_coh_as_env_of_E_0},  and for the canonical exact structure $\sS_{\textrm{can}}$, we have that $\aA_r(\Rref(X^{\aff}), \sS_{\textrm{can}}) \simeq \Coh^{\geq 1}(X^{\aff})$ is the quotient of $\Coh(X^{\aff})$ by the Serre subcategory of objects of finite length, according to Proposition \ref{prop_Cal-Pir}. Finally, any object of $\Coh(X^{\aff})$ is the cokernel of a morphism in $\Rref(X^{\aff})$, \cite[Theorem 2.1]{SchrVezz}. Under the equivalence $\Psi^{-1}$ all of the above holds for the right  abelian envelopes of $\Rref(X)$.

		Let $\sS_g^{\Psi}$ be the exact structure on $\Rref(X)$ corresponding to $\sS_g$ under $\Psi$. Functor $\Psi$ induces an equivalence $\wt{\Psi} \colon \aA_r(\Rref(X), \sS_g^{\Psi}) \xrightarrow{\simeq} \Coh(X^{\aff})$. To determine $\wt{\Psi}$ on an object $T\in \aA_r(\Rref(X), \sS_g^{\Psi}) $ we present it as the cokernel of $\f\colon R_1\to R_2$ with $R_1, R_2 \in \Rref(X)$. Then $\wt{\Psi}(T)= \textrm{coker}(\Psi(\f))$. Note that $\oO_X$ is projective in the exact structure $\sS_g^{\Psi}$ as $\oO_{X^{\aff}} = \Psi(\oO_X)$ is projective in $\sS_g$. Since $\wt{\Psi}(-)$ and $\Hom_{\aA_r(\Rref(X), \sS_g^{\Psi})}(\oO_X,-)$ are both right exact functors which coincide on objects of $\Rref(X)$, it follows that they are equivalent. 
		
	Finally, as $\Coh^{\geq 1}(X)$ is the quotient of $\aA_r(\Rref(X), \sS^{\Psi}_g)$ by the Serre subcategory of objects of finite length and the equivalence $\wt{\Psi}$ maps objects of finite length to objects of finite length, one concludes that $\wt{\Psi}$ induces an equivalence $\Hom_{\Coh^{\geq 1}(X)}(\oO_X, -) \colon \Coh^{\geq 1}(X) \xrightarrow{\simeq} \Coh^{\geq 1}(X^{\aff})$.
		
		Any divisor $D\subset X$ yields a Serre subcategory $\Coh^{\geq 1}_D(X)$ in $\Coh^{\geq 1}(X)$. Hence, the fact that ${\rho_X}_*\colon \Coh^{\geq 1}(X) \to \Coh^{\geq 1}(X^{\aff})$ is an equivalence implies that ${\rho_X} \colon X\to X^{\aff}$ does not contract any divisor, i.e. $\rho_X$ is quasi-finite.
		
		Since $X$ is separated and of finite-type, so is $\rho_X$. Finally, by construction, ${\rho_X}_*\oO_X = \oO_{X^{\aff}}$. By \cite[Theorem 12.83]{GorWed}, $\rho_X$ restricted to the open set $V = \{x\in X\,|\, \dim \rho_X^{-1} \rho_X(x)  = 0\}$ is an open embedding. Hence, as $\rho_X$ is quasi-finite, it is an isomorphism of $X$ with an open subscheme of $X^{\aff}$. Since $\rho_X$ induces an equivalence of $\Coh^{\geq 1}(X)$ and $\Coh^{\geq 1}(X^{\aff})$, the complement to $X$ in $X^{\aff}$ is of codimension 2, see Theorem \ref{thm_cal_pir}.
	\end{proof}

	\begin{THM}\label{thm_CM_aff=q-aff}
		A normal surface $X$ is quasi-affine if and only if it is CM affine.
	\end{THM}
	\begin{proof}
		By Proposition \ref{prop_affin_of_CM_affine} if $X$ is CM affine, it is quasi-affine. 
		
		By Proposition \ref{prop_aff_of_quasi-affine}, if $X$ is quasi-affine then $\rho_X \colon X\to X^{\aff}$ is an open embedding with complement of codimension two. Hence, $\rho_X$ induces an equivalence $\Rref(X) \simeq \Rref(X^{\aff})$, see Theorem \ref{thm_iso_iff_E_n-2}, and $X$ is CM affine.
	\end{proof}

\vspace{0.3cm}
\subsection{Construction of $\Coh(\wh{X})$ from $\Rref(X)$}\label{ssec_cont_of_X}~\\

Recall that $X^1$ denotes the set of prime divisors in $X$. For $I\subset X^1$, let $U_I\subset X$ be the complement to the divisor $\bigcup_{D\in I} D$.
\begin{LEM}\label{lem_another_def_of_X_hat}
	The codim-2-saturated model $\wh{X}$ is isomorphic to $\varinjlim \wh U_I^{\aff}$, where the colimit is taken over the category
	\begin{equation}\label{eqtn_set_for_colim}
	\mathcal{D}\textrm{iv-}\mathcal{QA}(X)= \{I\subset X^1\,|\, I \textrm{ is finite, } U_I \textrm{ is quasi-affine}\}.
	\end{equation}
\end{LEM}
\begin{proof}
	By \cite[Proposition 2.5.2]{KasSch2} it suffices to check that the functor 
	$\mathcal{D}\textrm{iv-}\mathcal{QA}(X) \to \mathcal{QA}(X)$
	 which sends $I$ to $U_I$ is cofinal, i.e. that for any quasi-affine open subset $U\subset X$ the category $(\mathcal{D}\textrm{iv-}\mathcal{QA}(X))^U$ is connected. We show that even more is true, i.e. that $(\mathcal{D}\textrm{iv-}\mathcal{QA}(X))^U$ has an initial object. 
	 	Let $D = \bigcup_{i\in I} D_i$ be the divisorial part of $X\setminus U$. Clearly, any inclusion $U \to U_J$, for  $J\in \mathcal{D}\textrm{iv-}\mathcal{QA}(X)$, factors via the inclusion $U\to X\setminus D$. By Proposition \ref{prop_map_to_X_hat}, $U_I = X \setminus D$ is an open subset of $U^{\aff}$, hence it is quasi-affine, i.e. $I \in \mathcal{D}\textrm{iv-}\mathcal{QA}(X)$.
\end{proof}

Consider a quasi-abelian category $\eE$ and the right abelian envelope $\aA$ of $\eE$ with its canonical exact structure $\sS_{\textrm{can}}$. We say that $\eE$ is \emph{height-one} if the torsion part $\tT = {}^\perp \eE$ of the canonical torsion pair in $\aA$ (see Proposition \ref{prop_cotilt_is_quasi-ab}) is the Serre subcategory of objects of finite length in $\aA$. 

If $C$ is a reduced curve then the category of torsion-free sheaves on $C$ is height-one.  By Corollary \ref{cor_env_of_E_n-2} and  Proposition \ref{prop_canon_torsion_pair_Coh_1}, the category $\Rref(X)$, for a normal surface $X$, is height-one.

Given a finite set  $I$ of simple objects in $\aA$ and the Serre subcategory $\sS_I$ generated by them, we denote by $\fF_I\subset \aA/\sS_I$ the right orthogonal to $\tT/\sS_I \subset \aA/\sS_I$.

Let $\eE$ be a height-one category and $\eE^1$ the set of isomorphism classes of simple objects in $\aA_r(\eE, \sS_{\textrm{can}})$. Recall that, for a category $\cC$, its center is denoted by  $Z(\cC)$. Define the \emph{saturated model} of $\eE$ by
	\begin{equation}\label{eqtn_def_of_E_hat}
\wh{\eE} := \underleftarrow{\textrm{2-lim}}\,
\textrm{mod}_{\textrm{fg}}\textrm{--} Z(\fF_I),
\end{equation}	
where the 2-limit in the 2-category of categories is taken over 
\begin{equation}\label{eqtn_diagram_from_Ref}
\{I\subset \eE^1\,|\, I \textrm{ is finite, } \fF_I \textrm{ is CM-affine}\}.
\end{equation}
\begin{THM}\label{thm_final_model_from_Ref}
	For a normal surface $X$, the saturated model $\wh{\Rref(X)}$ is equivalent to $\Coh(\wh{X})$.
\end{THM} 
\begin{proof}
	By Lemma \ref{lem_another_def_of_X_hat}, $\wh{X} = \varinjlim \wh{U}$, where the colimit is taken over \eqref{eqtn_set_for_colim}. Since a coherent sheaf can be glued along an open covering, $\Coh(\wh{X}) = \underleftarrow{\textrm{2-lim}}\, \Coh(\wh{U})$, where the 2-limit in the 2-category of categories is also taken over \eqref{eqtn_set_for_colim}.

	The category \eqref{eqtn_diagram_from_Ref}, for $\eE = \Rref(X)$, is isomorphic to (\ref{eqtn_set_for_colim}). Indeed, by Proposition \ref{prop_canon_torsion_pair_Coh_1},
		$X^1$ is the set of isomorphism classes of simple objects in $\aA_r(\Rref(X)) \simeq \Coh^{\geq 1}(X)$ (see Corollary \ref{cor_env_of_E_n-2}). Propositions \ref{prop_canon_torsion_pair_Coh_1} and \ref{prop_Ref_U_from_Ref_X} imply that,
		for a finite set $I\subset X^1$, the category $\fF_I$ is equivalent to $\Rref(U_I)$, where $U_I$ denotes the open complement to $\bigcup_{D_i \in I} D_i \subset X$. In view of Theorem \ref{thm_CM_aff=q-aff}, $U_I$ is quasi-affine if and only if $\Rref(U_I)$ is CM-affine. 
	
	Finally, by Proposition \ref{prop_cent_of_E_0}, the affinisation $U_I^{\aff}$ of $U_I$ is the spectrum of the center $H^0(U_I, \oO_{U_I})$ of $\Rref(U_I)$. Hence, $\wh{\Rref(X)} = \underleftarrow{\textrm{2-lim}}\, \Coh(U_I^{\aff})$, where the limit is taken over (\ref{eqtn_set_for_colim}). The statement follows.
\end{proof}

We define the \emph{geometric model} of a height-one quasi-abelian $\eE$ category by
\begin{equation}\label{eqtn_geom_mode}
\ol{\eE} = \varinjlim \Spec Z(\fF_I)
\end{equation}
where the colimit is taken over \eqref{eqtn_diagram_from_Ref}.

Theorem \ref{thm_final_model_from_Ref} gives a  way to recover the category $\Coh(\wh{X})$ of coherent sheaves on the codim-2-saturated model of $X$ from the category $\Rref(X)$. As a corollary, we can recover $\wh{X}$ itself as a geometric model of $\Rref(X)$.

\begin{COR}\label{cor_reconst_of_X}
	Given a normal surface $X$, $\ol{\Rref(X)}$ is isomorphic to $\wh{X}$. In particular, for a codim-2-saturated surface $X$, we have: $X\simeq \ol{\Rref(X)}$.
\end{COR}

\begin{proof}
	As in the proof of Theorem \ref{thm_final_model_from_Ref}, category $\fF_I$ is equivalent to $\Rref(U_I)$, hence $Z(\fF_I) = \Spec H^0(U_I, \oO_{U_I}) = U_I^{\aff}$, see Proposition \ref{prop_cent_of_E_0}. By the proof of Theorem \ref{thm_final_model_from_Ref}, the diagram category (\ref{eqtn_diagram_from_Ref}) is equivalent to (\ref{eqtn_set_for_colim}). The statement follows from Lemma \ref{lem_another_def_of_X_hat}.
\end{proof}
Recall that for a normal surface $X$, $\Rref(X)$ is equivalent to the category $\textrm{CM}(X)$ of maximal Cohen-Macaulay sheaves on $X$. For a reduced curve $C$, $\textrm{CM}(C)$ is equivalent to the category $\fF(C)$ of torsion-free sheaves on $C$, \emph{cf.}  \cite[Proposition 1.5]{Yoshino}.

Then we have an analogous statement for curves. 

\begin{PROP}
	Let $C$ be a reduced curve. 
	Then $C\simeq \ol{\fF(C)}$.
\end{PROP}
\begin{proof}
Category $\Coh(C)$ admits a cotiliting torsion pair $(\tT(C), \fF(C))$,	where the category $\tT(C)$ of torsion sheaves is the Serre subcategory of objects of finite length. Hence, $\Coh(C)$ is the right abelian envelope of $\fF(C)$ with the canonical exact structure by Proposition \ref{prop_cotilt_is_quasi-ab}.

The curve $C$ is affine if and only if $\fF(C)$ is CM-affine. Indeed, similar to the proof of Theorem \ref{thm_CM_aff=q-aff} one can show that the affinisation $C^{\aff}$ is again a curve. Then the equivalence $\fF(C) \simeq \fF(C^{\aff})$ induces an equivalence of the right abelian envelopes $\Coh(C) \simeq \Coh(C^{\aff})$, hence an isomorphism $C\simeq C^{\aff}$ (see Theorem \ref{thm_Rou}).

It follows that the limit in \eqref{eqtn_geom_mode} is taken over affine open subsets of $C$. Hence the statement.
\end{proof}

	\appendix
	
	\section{Abelian envelopes of quasi-abelian categories}\label{sec_ab_env_of_q-a_cat}
	
	We discuss exact categories and their right abelian envelopes. The  main example is that  of quasi-abelian categories endowed with their canonical exact structures. We recall the notion of cotilting torsion pairs and prove that, via the right abelian envelopes, they correspond to quasi-abelian categories.
	
	\vspace{0.3cm}
	\subsection{Exact categories and their right abelian envelopes}~\\
	
	Recall \cite{Quillen, Kel4} that an \emph{exact category}  is an additive category $\eE$ together with a fixed class $\sS$ of \emph{conflations}, i.e. pairs of composable morphisms
	\begin{equation}\label{eqtn_confl}
	X \xrightarrow{i} Y \xrightarrow{d} Z
	\end{equation} 
	such that $i$ is the kernel of $d$ and $d$ is the cokernel of $i$. We shall say that $i$ is an \emph{inflation} and $d$ a \emph{deflation}. The class $\sS$ is closed under isomorphisms and the pair $(\eE, \sS)$ is to satisfy the following axioms:
	\begin{itemize}
		\item[(Ex 0)] $0 \to X \xrightarrow{\Id_X} X$ is a conflation, for any object $X$ in $\eE$,
		\item[(Ex 1)] the composite of two deflations is a deflation,
		\item[(Ex 2)] the pullback of a deflation against an arbitrary morphism exists and is a deflation,
		\item[(Ex 2')] the pushout of an inflation along an arbitrary morphism exists and is an inflation.
	\end{itemize}
	
	We sometimes refer to the choice of $\sS$ as the choice of an \emph{exact structure} on $\eE$. Any additive category $\eE$ admits the \emph{split exact structure} $(\eE, \sS_{\textrm{triv}})$ with $\sS_{\textrm{triv}}$ consisting of split short exact sequences.
	
	Let $\aA$ be an abelian category and $\eE\subset \aA$ a full subcategory closed under extensions. Then short exact sequences in $\aA$ with all terms in $\eE$ define an exact structure on $\eE$ which will be referred to as the \emph{induced exact structure}. 
	
	For objects $E_1$ and $E_2$ in an exact category $(\eE, \sS)$, we denote by $\Ext^1_{(\eE, \sS)}(E_1,E_2)$ the abelian group whose elements are classes of conflations
	$E_2\to X \to E_1$ considered up to isomorphisms which are identical on $E_1$ and $E_2$.  If the exact structure on $\eE$ is clear from the context we sometimes omit it from the notation.

	Let $(\eE, \sS)$ be an exact category and $\aA$ an abelian one. Recall that an additive functor $F\colon \eE \to \aA$ is \emph{right exact} if for any conflation $X \xrightarrow{f} Y \xrightarrow{g} Z$ in $\sS$, the sequence $F(X) \xrightarrow{F(f)}F(Y) \xrightarrow{F(g)} F(Z) \to 0$ is exact. We denote by $\textrm{Rex}(\eE, \aA)$ the category whose objects are right exact functors $\eE\to \aA$. Morphisms in $\textrm{Rex}(\eE, \aA)$ are natural transformations.
	
	Recall that an abelian category $\aA_r(\eE)$ together with a right exact functor $i_R\colon \eE\to \aA_r(\eE)$ is the  \emph{right abelian envelope}  of $\eE$ \cite{BodBon4} if for any abelian category $\aA'$ the composition with $i_R$ yields an equivalence $\textrm{Rex}(\aA_r(\eE), \aA') \xrightarrow{\simeq} \textrm{Rex}(\eE, \aA')$. 
	
	For an example of an exact category without the right abelian envelope see \cite[Example 4.5]{BodBon4}.
	The following criterion allows us to determine if an abelian category is the right abelian envelope of its full subcategory closed under extensions:
	\begin{PROP}\cite[Theorem 4.7]{BodBon4}\label{prop_r_ab_env}
		Let $\aA$ be an abelian category and $\eE\subset \aA$ a full subcategory closed under extensions and kernels of epimorphisms. If every object of $\aA$ is a quotient of an object of $\eE$ then $\aA\simeq \aA_r(\eE)$ is the right abelian envelope of $\eE$ with the induced exact structure.
	\end{PROP}
	
	\vspace{0.3cm}
	\subsection{Torsion pairs}\label{ssec_tros_pairs}~\\
	
	Recall that a \emph{torsion pair} $(\tT, \fF)$ in an abelian category $\aA$ is a pair of full subcategories $\tT, \fF$ such that $\Hom_{\aA}(\tT, \fF) =0$ and any object $A\in \aA$ fits into a short exact sequence $0\to T \to A \to F \to 0$, with $T\in \tT$, $F\in \fF$.
	
	Let $(\tT_1, \fF_1)$ be a torsion pair in an abelian category $\aA_1$ and $(\tT_2, \fF_2)$ a torsion pair in an abelian category $\aA_2$. An \emph{equivalence} of $(\tT_1, \fF_1)$ and $(\tT_2, \fF_2)$ is an equivalence $\phi \colon \aA_1 \xrightarrow{\simeq} \aA_2$ which restricts to an equivalence of subcategories $\phi|_{\fF_1}\colon \fF_1 \xrightarrow{\simeq} \fF_2$.
	
	Recall that a torsion pair $(\tT, \fF)$ in an abelian category $\aA$ is \emph{cotilting} if any object of $\aA$ is a quotient of an object of $\fF$.
	
	Given a torsion pair $(\tT, \fF)$ in $\aA$, both subcategories $\tT, \fF\subset \aA$ are closed under extensions,  $\tT\subset \aA$ is closed under quotients and $\fF \subset \aA$ -- under subobjects.
	
	\begin{LEM}\label{lem_ker_coker_in_F}\cite[Section 	5.4]{BvdB}
		Let $(\tT, \fF)$ be a torsion pair in an abelian category $\aA$. Then categories $\tT$ and $\fF$ have kernels and cokernels. 
	\end{LEM}
	The kernels in $\tT$ are the torsion parts of the kernels in $\aA$ while the cokernels are the cokernels in $\aA$. Dually, the kernels in $\fF$ are the kernels in $\aA$ while the cokernels are the torsion-free parts of the cokernels in $\aA$.
	
	Recall, that a morphism $f$ in an additive category with kernels and cokernels is \emph{strict} if the canonical map $\textrm{coker}\ker f \to \ker \textrm{coker} f$ is an isomorphism.
	
	\begin{PROP}\label{prop_strict_mono_epi_in_F}
		Let $(\tT, \fF)$ be a torsion pair in an abelian category $\aA$. Let $f$ be a morphism in $\fF$ and $t$ a morphism in $\tT$. Then
		\begin{enumerate}
			\item $f$ is a strict monomorphism if and only if it is a monomorphism in $\aA$ with cokernel in $\fF$. Further, $f$ is a strict epimorphism if and only if it is an epimorphism in $\aA$,
			\item $t$ is a strict monomorphism if and only if it is a monomorphism in $\aA$. Further, $t$ is a strict epimorphism if and only if it is an epimorphism in $\aA$ with kernel in $\tT$.
		\end{enumerate}
	\end{PROP}
	\begin{proof}
		By Lemma \ref{lem_ker_coker_in_F}, a morphism $f\colon F_1 \to F$ is a monomorphism in $\fF$ if and only if it is a monomorphism in $\aA$. Let $F\to Q$ be the cokernel of $f$ in $\aA$ and $F\to F_Q$ its cokernel in $\fF$. The snake lemma for diagram
		\[
		\xymatrix{0 \ar[r] & 0 \ar[r] & F_Q \ar[r] & F_Q \ar[r] & 0\\
			0 \ar[r] & F_1 \ar[r] \ar[u] &F \ar[r] \ar[u] & Q \ar[r] \ar[u] & 0}
		\] 
		implies that $\ker \textrm{coker} f= \ker(F \to F_Q)$ is isomorphic to $\textrm{coker} \ker f = \textrm{coker}(0 \to F_1) = F_1$ if and only if $T_Q = \ker (Q \to F_Q) \simeq  0$, i.e. $Q\simeq F_Q$ is an object in $\fF$.
		
		An epimorphism $f\colon F \to F_2$ in $\fF$ considered as a morphism in $\aA$ has cokernel $Q$ which is an object of $\tT$. Let $K = \ker_{\aA}f$. The short exact sequence $0 \to \textrm{Im}_{\aA}f \to F_2 \to Q\to 0$ in $\aA$ shows that  $\ker \textrm{coker} f = \ker (F_2 \to 0) = F_2$ is isomorphic to $\textrm{coker} \ker f = \textrm{coker}(K \to F) = \textrm{Im}_{\aA} f$ if and only if $Q =0$, i.e. $f$ is an epimorphism in $\aA$.
		
		The statement for strict mono- and epimorphisms in $\tT$ follows by considering the opposite categories.
	\end{proof}
	
	\vspace{0.3cm}
	\subsection{Quasi-abelian categories and cotilting torsion pairs}~\\

	Recall that an additive  category $\eE$ with kernels and cokernels is \emph{quasi-abelian} if the family of kernel-cokernel sequences defines an exact structure on $\eE$. We say that it is the \emph{canonical exact structure} and denote it by $(\eE, \sS_{\textrm{can}})$.
	
	The following theorem gives (the only) examples of quasi-abelian categories.
	
	\begin{PROP}\label{prop_t_tf_is_q-a}
		Let $(\tT, \fF)$ be a torsion pair in an abelian category $\aA$. Then
		\begin{enumerate}
			\item category $\fF$ is quasi-abelian and the canonical exact structure is the one induced from $\aA$,
			\item category $\tT$ is quasi-abelian and the canonical exact structure is the one induced from $\aA$.
		\end{enumerate}
	\end{PROP}
	\begin{proof}
		In a kernel-cokernel sequence both morphisms are strict. Hence, by Proposition \ref{prop_strict_mono_epi_in_F}, kernel-cokernel sequences $F_1 \xrightarrow{i} F \xrightarrow{d} F_2$, in $\tT$ or in $\fF$,  correspond to short exact sequences $0\to F_1 \xrightarrow{i} F \xrightarrow{d} F_2\to 0$ in $\aA$. It follows that the family of kernel-cokernel sequences defines an exact structure on $\tT$ or on $\fF$ and this exact structure  is induced from $\aA$.
	\end{proof}
	
	By \cite[Proposition 1.2.32]{Schne}	any quasi-abelian category $(\eE, \sS_{\textrm{can}})$ with the canonical exact structure admits the right abelian envelope. The category $\aA_r(\eE, \sS_{\textrm{can}})$ is defined as the heart of a \tr e in the derived category of $\eE$. More precisely, $\aA_r(\eE, \sS_{\textrm{can}})$ is the full subcategory of the homotopy category of $\eE$ with objects complexes $ 0 \to E_1\xrightarrow{d_E} E_0 \to 0$ given by monomorphisms $d_E$, localised  by the multiplicative system formed by homotopy classes of morphisms of complexes for which the square
	\[
	\xymatrix{ E_1 \ar[r]^{d_E} & E_0   \\
		F_1 \ar[r]^{d_F} \ar[u] & F_0 \ar[u] }
	\]
	is both cartesian and co-cartesian, \cite[Corollary 1.2.20]{Schne}.
	
	Let $\mathscr{Q}uasi\textrm{-}ab$ denote the groupoid of quasi-abelian categories with equivalences up to functorial isomorphisms as morphisms and $\mathscr{C}otilt$ the groupoid of cotilting torsion pairs with equivalences up to functorial isomorphisms as morphisms.
	Define 
	\begin{equation}\label{eqtn_bij_cotilit_q-a}
	\varphi \colon \mathscr{C}otilt \to \mathscr{Q}uasi\textrm{-}ab
	\end{equation}
	which maps a cotilting torsion pair $(\tT, \fF)$ to the category $\fF$. Proposition \ref{prop_t_tf_is_q-a} ensures that $\varphi$ is well-defined. 
	
	\begin{PROP}\label{prop_cotilt_is_quasi-ab} \emph{cf.} \cite[Proposition B.3]{BvdB}
		Functor $\varphi$ (\ref{eqtn_bij_cotilit_q-a}) is an equivalence. Its quasi-inverse maps a quasi-abelian category $\eE$ to the torsion pair $(\tT, \eE)$ in the right abelian envelope $\aA_r(\eE)$.
	\end{PROP}
	
	\section{Reconstructing a scheme $X$ from $\Coh(X)$ and $\QCoh(X)$}\label{sec_X_from_coh_X}
	
	For completeness we briefly recall how to reconstruct a Noetherian scheme $X$ from the category of coherent sheaves following \cite{Gabriel}, \emph{cf.} \cite{Rou3}.
	
	First, we argue that one can equivalently reconstruct $X$ from $\QCoh(X)$, as $\Coh(X)$
	and $\QCoh(X)$ determine each other.
	
	Recall \cite{KasSch2} that an \emph{ind-object} over a category $\cC$ is an object of $\textrm{Fun}(\cC^{\opp}, \textrm{Sets})$ which is isomorphic to an inductive limit of representable functors. By \cite[Corollary 8.6.3]{KasSch2} if $\cC$ is a small abelian category then  the category of ind-objects is equivalent to the category $\Lex(\cC^{\opp}, \Ab)$ of contravariant left exact functors from $\cC$ to the category of abelian groups.

Recall that an abelian category is \emph{locally Noetherian} if is has exact inductive limits and a family of Noetherian generators, \emph{cf.} \cite[Section II.4]{Gabriel}. By \cite[Theorem VI.2.1]{Gabriel} the category of $\oO_X$--modules on a  Noetherian scheme is locally Noetherian. Hence, so is its full subcategory $\QCoh(X)$ of quasi-coherent $\oO_X$--modules. 
	
	\begin{PROP}\label{prop_Coh_and_QCoh}
		Let $X$ be a Noetherian scheme. Then $\QCoh(X)$ is equivalent to the category of ind-objects over $\Coh(X)$ while $\Coh(X) \subset \QCoh(X)$ is equivalent to the subcategory of Noetherian objects.
	\end{PROP}
	\begin{proof}
		Let us first show that $\Coh(X) \subset \QCoh(X)$ is the subcategory of Noetherian objects. Let $F$ be a coherent sheaf and $X= \bigcup_{i=1}^n U_i$ an affine open cover of $X$. An ascending chain $\{F_i\}_i \subset F$ stabilizes when restricted to any $U_i$, as coherent modules over Noetherian rings are Noetherian (see \cite[Proposition 6.5]{AtiyahMacdonald}). Hence, $\{F_i\}_i$ stabilizes.
		
		Let now $F\in \QCoh(X)$ be a Noetherian object. We check that for any open inclusion $j\colon U \to X$, the sheaf $j^*F$ is also Noetherian. Given $\{G_i\}\subset j^*F$, $\{j_*G_i\}_i\subset j_*j^*F$ and $\{F_i := F \times_{j_*j^*F} j_*G_i\}\subset F$ are ascending chains,  where the fiber product is taken over the adjunction unit $F\to j_*j^*F$. As $F$ is Noetherian, the chain $\{F_i\}_i$ stabilizes. Hence, so does $\{G_i \simeq j^*F_i\}$. In particular, if $U\subset X$ is affine, $j^*F$ is Noetherian hence coherent. It follows that $F \in \Coh(X)$.
		
		By \cite[Theorem II.4.1]{Gabriel}, given a Noetherian category $\cC$, the category $\Lex(\cC^{\opp}, \Ab)$ is the unique up to an equivalence locally Noetherian category $\dD$ such that $\cC$ is equivalent to the category of Noetherian objects in $\dD$.  As the category $\QCoh(X)$ is locally Noetherian,  the above argument shows that $\QCoh(X) \simeq \Lex(\Coh(X)^{\opp},\Ab)$ is equivalent to the category of ind-objects over  $\Coh(X)$.
	\end{proof}

	The reconstruction of a scheme from the category of (quasi-)coherent sheaves relies on the notion of a spectrum of an abelian category $\aA$. 
	The \emph{spectrum} $\Spec(\aA)$ is the set of isomorphism classes of indecomposable injective objects in $\aA$. By \cite[Theorem VI.1.1]{Gabriel}, for a Noetherian scheme $X$, indecomposable injective objects in $\QCoh(X)$ are in bijection with the set of irreducible closed subsets of $X$ (to an irreducible closed $Z\subset X$ the bijection assigns the injective envelope of $\oO_Z$).

A	 localising subcategory $\bB$ of an abelian category $\aA$ allows one to represent $\Spec(\aA)$ as a disjoint union of $\Spec(\bB)$ and $\Spec(\aA/\bB)$, see \cite[Section IV.1]{Gabriel}.
We say that a localising subcategory $\bB\subset \aA$ is \emph{of finite type} if there exists an object $M \in \aA$ such that $\bB$ is the minimal localising subcategory containing $M$. 

To a  Grothendieck category $\aA$ one assigns the ringed space $R(\aA)$ with  $\Spec (\aA)$ as the underlying set. Open subsets in $R(\aA)$ are of the form $\Spec(\aA/\bB)$, for localising subcategories of finite type $\bB\subset \aA$ while the sheaf $\oO$ is defined via centers, $\oO(\Spec(\aA/\bB)):=Z(\aA/\bB)$. 

\begin{THM}\cite{Gabriel}\label{thm_Gab}
	A Noetherian scheme $X$ is isomorphic to $R(\QCoh(X))$.
\end{THM}

If the scheme $X$ is separated of finite type over $k$, one can reconstruct $X$ from $\Coh(X)$ directly, \emph{cf.} \cite{Rou3}.

Define a Serre subcategory $\bB\subset \aA$ to be \emph{of finite type} if it is the smallest Serre subcategory in $\aA$ containing $M$, for some object $M\in \aA$. It is \emph{irreducible} if it is not equal to zero and it is not generated by two proper Serre subcategories.

By \cite[Theorem 4.1]{Rou3}, the map $Z \mapsto \Coh_Z(X)$ is a bijection between closed subsets $Z\subset X$ and 
Serre subcategories of finite type in $\Coh(X)$. Here $\Coh_Z(X)\subset \Coh(X)$ is the full subcategory of sheaves supported on $Z$. 

To an abelian category $\aA$ one assigns the ringed space $r(\aA)$. The underlying set is the set of irreducible Serre subcategories of finite type in $\aA$. Open subsets in $r(\aA)$ are the sets $D(\bB)$ of points of $r(\aA)$ not contained in a given Serre subcategory $\bB\subset \aA$ of finite type while the sheaf $\oO$ is defined as $\oO(D(\bB)) := Z(\aA/\bB)$.

\begin{THM}\cite{Rou3}\label{thm_Rou}
	A separated Noetherian scheme $X$ of finite type over $k$ is isomorphic to $r(\Coh(X))$.
\end{THM}

	\bibliographystyle{alpha}
		\bibliography{../../ref}

\end{document}